\documentclass[11pt]{amsart}

\usepackage[utf8]{inputenc}
\usepackage[T1]{fontenc}
\usepackage{lmodern}
\usepackage{amsmath, amssymb, mathtools, amscd}
\usepackage{stmaryrd}
\usepackage{bbm}
\usepackage{tabularx}
\usepackage{longtable}
\usepackage{enumitem}
\usepackage[dvipsnames]{xcolor}
\usepackage{tikz}
\usetikzlibrary{decorations.markings}
\usepackage[all]{xy}
\usepackage{adjustbox}
\usepackage{graphicx}
\usepackage{textcomp}
\usepackage[normalem]{ulem}
\usepackage{cancel}
\usepackage{scalerel}

\DeclareFontFamily{U}{mathx}{}
\DeclareFontShape{U}{mathx}{m}{n}{<-> mathx10}{}
\DeclareSymbolFont{mathx}{U}{mathx}{m}{n}
\DeclareMathAccent{\widehat}{0}{mathx}{"70}
\DeclareMathAccent{\widecheck}{0}{mathx}{"71}
\DeclareMathAccent{\widebar}{0}{mathx}{"73}

\theoremstyle{plain}
\newtheorem{theorem}{Theorem}[subsection]
\newtheorem{lemma}[theorem]{Lemma}
\newtheorem{proposition}[theorem]{Proposition}
\newtheorem{corollary}[theorem]{Corollary}

\theoremstyle{definition}
\newtheorem{definition}[theorem]{Definition}

\newtheorem{remark}[theorem]{Remark}

\newtheorem{construction}[theorem]{Construction}

\makeatletter
\@namedef{subjclassname@1991}{2020 Mathematics Subject Classification}
\makeatother

\usepackage{hyperref}

\makeatletter
\renewcommand{\subsection}{\@startsection{subsection}{2}%
  \z@{.5\linespacing\@plus.7\linespacing}{.5\linespacing}%
  {\normalfont\bfseries}}
\makeatother
\hypersetup{colorlinks=true,linkcolor=blue,citecolor=red,urlcolor=cyan,
pdfauthor={Zhangqi Pan},pdftitle={A Note on n-Representations of Simplicial Groups}}
\title[A Note on $n$-Representations]{A Note on $n$-Representations of Simplicial Groups}
\author{Zhangqi Pan}
\address{School of Mathematics and Statistics\\Huazhong University of Science and Technology}
\email{2434271783@qq.com}
\date{\today}
\begin{document}
\nocite{Kerodon,HTT,HA,SchwedeShipley,GoerssJardine,Arakawa2023,Weibel,Kelly1982,Emir2021,BaezLauda2004,BaezCrans2004,HerediaElgueta2016,Huan2026,Elgueta2007,Elgueta2011,AriasAbadCrainic2013,MutluPorter1998,ThomasExtended,Ngompe2024,HatcherSpectralSequences}
\begin{abstract}
For $n\ge1$, we study representations of a simplicial group $G$ in the
full $\infty$-subcategory of simplicial vector spaces whose normalized
homology is concentrated in degrees $0,\ldots,n-1$, retaining its full
mapping spaces. The resulting functor category depends only on the
$n$-truncated homotopy type of $\mathbf BG$ and admits projectively
bifibrant strict diagram and simplicial group-algebra module models.
If restriction along a specified map $f:G\to H$ is an equivalence,
then $\pi_0(f)$ is a group isomorphism; the proof uses ordinary extension
and restriction of scalars. Yet for every $n\ge2$ and prime $q$
invertible in $k$, the nontrivial source $K(\mathbb Z/q,n-1)$ has the
same $n$-representation category with vertex evaluation as the trivial
group. The group-algebra augmentation is a weak equivalence, so the
constant functor is also an equivalence for untruncated representations.
Thus $n$-representations depend only on the $n$-truncation of
$\mathbf BG$, but need not determine that truncation, even together
with vertex evaluation.
\end{abstract}
\maketitle
\setcounter{tocdepth}{1}
\tableofcontents
\section{Introduction}\label{sec:introduction}

Representations of a higher group depend on the choice of a linear
target as well as on the homotopy type of the source. We study this
dependence for simplicial groups acting on simplicial vector spaces of
finite homological height. Two questions guide the discussion: which
maps of sources induce equivalences of representation categories, and
what such an equivalence implies about the source homotopy groups.

Fix a field $k$. Let $\mathcal V^\infty$ be the $\infty$-category presented by simplicial
$k$-vector spaces. For $n\ge1$, let
\[
\mathbf{Vect}^{(n)}_k\subseteq\mathcal V^\infty
\]
be the full subcategory on objects whose Dold--Kan normalized homology
is concentrated in degrees $0,\ldots,n-1$.
We retain its full mapping spaces; Proposition~\ref{prop:n-linear-category-height}
shows that they are already $(n-1)$-truncated. Write $\mathbf BG$ for
the one-vertex classifying space of a simplicial group $G$. We define
\[
\mathbf{Rep}^{(n)}_k(G):=
\operatorname{Fun}(\mathbf BG,\mathbf{Vect}^{(n)}_k).
\]
This definition applies to every simplicial group $G$.
Evaluation at the unique vertex of $\mathbf BG$ is the forgetful functor
to $\mathbf{Vect}^{(n)}_k$: it returns the underlying coefficient object.
For $n=1$, this representation $\infty$-category is equivalent to the
nerve of the ordinary representation category of $\pi_0G$ on all
$k$-vector spaces, compatibly with the forgetful functors.
For $n=2$, coefficient objects admit two-term chain-complex models
$V_1\to V_0$ up to quasi-isomorphism.

For this target, source truncation and source detection have different
answers. Representations depend only on the $n$-truncated homotopy type
of $\mathbf BG$. Conversely, an equivalence induced by a specified
group map forces an isomorphism on $\pi_0$, but need not force one on
$\pi_{n-1}$ when $n\ge2$. We call the latter phenomenon a
\emph{coefficient obstruction}: a coefficient-dependent obstruction
to detecting $\pi_{n-1}$. In the examples considered here, the
obstruction already occurs for untruncated representations. The finite
height of the target and the choice of coefficient field thus impose
distinct restrictions on source detection.

\subsection{Main results}\label{sec:main-results}
\begin{itemize}
\item \textbf{Source invariance and algebraic models.}
Theorems~\ref{thm:source-truncation-invariance} and~\ref{thm:A1}
show that $\mathbf{Rep}^{(n)}_k(G)$ depends only on the $n$-truncated
homotopy type of $\mathbf BG$. In particular, if a specified map
$f:G\to H$ induces isomorphisms on $\pi_0,\ldots,\pi_{n-1}$, then
\[
f^*:\mathbf{Rep}^{(n)}_k(H)\xrightarrow{\simeq}\mathbf{Rep}^{(n)}_k(G).
\]
Section~\ref{sec:truncated-invariance} also constructs an explicit
simplicial-group model for the $(n-1)$-truncation of the source.

To relate this definition to explicit algebraic data,
Proposition~\ref{prop:n-representation-function-model} and
Corollary~\ref{cor:n-representation-mapping-model} present this
$\infty$-category by the homotopy coherent nerve of the full simplicial
subcategory of projectively bifibrant strict diagrams whose value at the
unique vertex has the prescribed homology range.
The underlying simplicial sets of their enriched Hom objects compute
the mapping spaces in $\mathbf{Rep}^{(n)}_k(G)$.
Theorem~\ref{thm:module-bridge} identifies the strict enriched diagram
category with simplicial $k[G]$-modules and compares their derived models;
Corollary~\ref{cor:n-representation-module-model} states its finite-height
form explicitly.
\item \textbf{Detection and coefficient obstructions.}
Theorem~\ref{thm:A4} proves that, if restriction along the given
group map $f$ is an equivalence, then $\pi_0(f)$ is a group isomorphism.
The proof reduces to all ordinary modules over group algebras and the
extension--restriction adjunction.
For $n=1$, the two directions say that this restriction functor is an
equivalence if and only if $\pi_0(f)$ is an isomorphism.

For positive homotopy degrees, Proposition~\ref{prop:coefficient-acyclic-invisibility}
gives a sufficient condition for a source to have the same
representation category as the trivial group: if
$\widetilde H_i(|G|;k)=0$ for every $i\ge0$, then the specified constant
functor is an equivalence
$\mathcal V^\infty\simeq\mathcal R_k(G)$, inverse to evaluation.
It restricts to the finite-height subcategories at every height.
For every $n\ge2$ and every prime $q$ invertible in $k$,
Corollary~\ref{thm:prime-to-character-invisibility} applies this criterion
to $K(\mathbb Z/q,n-1)$ and gives an equivalence, induced by its map to $*$,
\[
\mathbf{Vect}^{(n)}_k
\xrightarrow{\ \simeq\ }
\mathbf{Rep}^{(n)}_k(K(\mathbb Z/q,n-1))
\]
compatible with evaluation, although the source has nonzero
$\pi_{n-1}\cong\mathbb Z/q$. Thus dependence on the truncated source
does not imply that the representation category, even together with
evaluation, determines its $(n-1)$-st homotopy group.
\end{itemize}

Dold--Kan provides the mapping-space calculation that determines the
height of the coefficient target. Source invariance then follows from
the universal property of homotopy $n$-categories. The algebraic models
use the strictification theorem
\cite[Proposition~4.2.4.4]{HTT}, with the comparison restricted according
to the homology of the evaluated object. For coefficient-acyclic sources,
the augmentation $k[G]\to k$ is a weak equivalence; a natural zigzag of
module weak equivalences identifies constant representations and
evaluation as inverse functors. The module equivalence reflects
\cite[Proposition~4.2.4.9]{HA}: an algebra whose unit is an equivalence
has a module category equivalent to the underlying coefficient category.
We use these comparisons to express source invariance and detection in
both functor and module models, retaining the specified restriction and
evaluation functors.

\subsection{Relation to higher representation theory}
\label{sec:higher-representation-literature}

The closest low-dimensional comparison is with representations in
Baez--Crans $2$-vector spaces. At height two, Dold--Kan models our
coefficient objects by two-term chain complexes $V_1\to V_0$, equivalently
by internal categories in vector spaces \cite{BaezCrans2004}.
Heredia--Elgueta study representations of $2$-groups in these coefficients
\cite{HerediaElgueta2016}. For a $2$-group with finite $\pi_0$ and
$\pi_1$ and a split $k$-linear abelian coefficient category, their
Theorem~4.3 shows, when the
characteristic is zero or coprime to the orders of $\pi_0$ and $\pi_1$, that
the homotopy category of their representation $2$-category is equivalent to
two copies of the representation category of $\pi_0$; it also describes
the $2$-cells between the resulting reduced representations. Their work
thus exhibits a related failure of positive-degree detection. Here we
work with the functor $\infty$-category defined above, and obtain the
evaluation-compatible equivalence for $K(\mathbb Z/q,n-1)$ at every
height $n\ge2$ from an untruncated coefficient-homology criterion.
Section~\ref{sec:kerodon-direct} discusses the two coefficient models.

Other linear targets lead to different representation data.
With the Kapranov--Voevodsky coefficient target
over $\mathbb C$, Elgueta's classification of nonzero
$2$-representations uses $\pi_0$, its module $\pi_1$, and the Postnikov
class in $H^3(\pi_0,\pi_1)$, together with representation data subject
to a cocycle condition \cite{Elgueta2007}. For essentially finite $2$-groups,
his later regular $2$-representation represents the forgetful
$2$-functor under algebraic-closure and characteristic-zero hypotheses
\cite[Theorem~30]{Elgueta2011}. These results use a different linear
target from the two-term complex model above.

Geometric variants also use functors from a delooping. Over a
characteristic-zero field, Huan defines $2$-representations of coherent
Lie $2$-groups by pseudofunctors into the $2$-category of $k$-prelinear
categories and proves invariance under biequivalence
\cite[Definition~2, Remark~2, and Theorem~2]{Huan2026}.
Representations up to homotopy of Lie groupoids use bounded graded
vector bundles over the object manifold, equipped with a square-zero
operator on groupoid cochains \cite{AriasAbadCrainic2013}.
These theories concern categorical or bundle-valued coefficients,
whereas the present paper fixes simplicial $k$-vector-space coefficients.

\subsection{Scope and outline}\label{sec:outline}
We use the finite-height coefficient target as an $\infty$-category.
For $n\ge2$, the inherited levelwise tensor does not preserve its
homology range; Section~\ref{sec:truncated-rep-spaces} gives the explicit
calculation.

Section~\ref{sec:preliminaries} supplies the mapping-space and truncation
tools. Section~\ref{sec:classifying-spaces} defines the representations and
proves source invariance. Section~\ref{sec:rep-module-categories} constructs
strict diagram and module models. Section~\ref{sec:detection} proves
degree-zero detection and the coefficient obstruction.

\section{Homotopical and Enriched Foundations}\label{sec:preliminaries}

This section collects the notation and standard facts used throughout
the paper.
Simplicial groups and their classifying spaces are small in a fixed
universe, while coefficient and
functor categories are formed in a fixed larger universe.

\subsection{Simplicial sets and simplicial groups}\label{sec:simplicial-sets-groups}

Let $\mathsf{sSet}$ denote simplicial sets, equipped throughout with the
Kan--Quillen model structure. Its cofibrations are degreewise
monomorphisms, its fibrations have the right lifting property against
all horn inclusions, and its weak equivalences induce weak homotopy
equivalences on geometric realizations.

\begin{definition}\label{def:simplicial-group}
A simplicial group is a simplicial object
\[
G:\Delta^{\mathrm{op}}\longrightarrow\mathsf{Grp}.
\]
Equivalently, it consists of groups $G_m$ with group-homomorphic face
and degeneracy maps satisfying the simplicial identities.
\end{definition}

Write $\mathsf{sGrp}$ for the category of simplicial groups; weak
equivalences are detected on their underlying simplicial sets.

\begin{proposition}\label{prop:simplicial-group-kan}
For every simplicial group $G$, the underlying simplicial set of $G$ is
a Kan complex.
\end{proposition}

This is Moore's theorem
(\cite[Lemma I.3.4]{GoerssJardine}; see also
\cite[Proposition~1.2.5.9, Tag~\href{https://kerodon.net/tag/00MG}{00MG}]{Kerodon}).

\begin{definition}[homotopy $n$-groups]\label{def:homotopy-n-group}
For $n\ge1$, a simplicial group $G$ is called a \emph{homotopy
$n$-group} if
\[
\pi_i(G,e)=0\qquad(i\ge n),
\]
where $e$ is the unit vertex.
\end{definition}

\begin{remark}[low-dimensional strict models]
\label{rem:homotopy-n-group-strict-models}
The homotopy-group condition in the preceding definition allows us to
replace $G$ by a weakly equivalent algebraic model. In low dimensions,
ordinary groups and crossed modules provide such models.

For an ordinary group $\Gamma$, the \emph{constant simplicial group}
$\operatorname{const}\Gamma$ has
\[
(\operatorname{const}\Gamma)_m=\Gamma\quad(m\ge0),
\qquad d_i=\operatorname{id}_\Gamma,
\qquad s_i=\operatorname{id}_\Gamma
\]
for every admissible face and degeneracy map. Its underlying simplicial
set is a disjoint union of points indexed by $\Gamma$: its components
form the group $\Gamma$, and its positive-degree homotopy groups vanish.
For any simplicial group $G$, the natural quotient
\[
q:G\longrightarrow\operatorname{const}\pi_0(G)
\]
sends a simplex to the component of any of its vertices; all its vertices
belong to the same component. These maps are group homomorphisms and
commute with faces and degeneracies, and $\pi_0(q)$ is the identity
under the canonical identification of the target's components with
$\pi_0(G)$. If $G$ is a homotopy $1$-group, both source and target have
zero positive-degree homotopy groups, so $q$ induces isomorphisms on all
homotopy groups and is a weak equivalence. Thus
$\operatorname{const}\pi_0(G)$ models the homotopy type of $G$.

For $n=2$, one can choose a natural strict model.  With the Moore convention
$N_mG=\bigcap_{0\leq i<m}\ker d_i$ and boundary $d_m$, there is a natural map
$G\to P_1G$ for which $N_m(P_1G)=1$ when $m\geq2$ and which induces
isomorphisms on $\pi_0$ and $\pi_1$.  Its associated crossed module has
$P=G_0$, $E=N_1G/d_2(N_2G)$, and boundary induced by $d_1$.  The target has
no higher homotopy. If $G$ is a homotopy $2$-group, its higher homotopy
groups also vanish; together with the displayed isomorphisms on $\pi_0$
and $\pi_1$, this makes $G\to P_1G$ a weak equivalence, with the
Moore-length-one condition imposed on the model $P_1G$.
This follows from the truncation--coskeleton
construction of \cite[Sections~2.4 and 2.7--2.8]{ThomasExtended}, after
reversing its Moore face convention.
Recall that a crossed module is a homomorphism $\partial:E\to P$ together with
an action of $P$ on $E$ satisfying
\[
\partial(p\cdot e)=p\partial(e)p^{-1},
\qquad
\partial(e)\cdot f=efe^{-1}.
\]
If $G$ has Moore length one, namely $N_mG=1$ for $m\geq2$, then
\[
N_1G=\ker(d_0:G_1\to G_0),
\qquad
\partial=d_1:N_1G\longrightarrow G_0,
\qquad
g\cdot e=s_0(g)e s_0(g)^{-1}
\]
is a crossed module.  In the reverse construction the groups in degrees zero
and one are $P$ and $E\rtimes P$, with
$d_0(e,p)=p$, $d_1(e,p)=\partial(e)p$, and $s_0(p)=(1,p)$.
These constructions give an equivalence of categories between crossed modules
and Moore-length-one simplicial groups
\cite[Theorem~2.2]{MutluPorter1998}; the displayed Moore convention and
crossed-module formulas agree with
\cite[Section~1, pp.~190--191]{Emir2021}.  A crossed module also determines
a strict $2$-group, viewed as an internal category in groups: its object group
is $P$, its arrow group is $E\rtimes P$, and its source, target, identity,
and composition are
\[
s(e,p)=p,\quad t(e,p)=\partial(e)p,\quad i(p)=(1,p),\quad
(f,\partial(e)p)\circ(e,p)=(fe,p).
\]
Conversely, for an internal category $C$ in groups one takes
$P=C_0$, $E=\ker(s:C_1\to C_0)$, $\partial=t|_E$, and the action given by
conjugation by identity arrows.  These constructions are inverse up to
natural isomorphism \cite[Section~7, pp.~462--463]{BaezLauda2004}.
The order $fe$ in the displayed composition is forced by its target;
the factors appear reversed in the cited printed formula.
\end{remark}

\subsection{Simplicial vector spaces}\label{sec:simplicial-vector-spaces}

Throughout this paper, $k$ denotes a fixed field.

\begin{definition}\label{def:simplicial-vector-space-object}
\label{def:category-svect}
A simplicial $k$-vector space is a functor
$V:\Delta^{\mathrm{op}}\to\mathsf{Vect}_k$, where $\Delta$ is the category
of finite nonempty ordinals and order-preserving maps, and
$\mathsf{Vect}_k$ is the category of $k$-vector spaces. Equivalently,
it consists of $k$-vector spaces $V_n$ for $n\geq0$, linear face maps
$d_i:V_n\to V_{n-1}$ for $n\geq1$ and $0\leq i\leq n$, and linear
degeneracy maps $s_i:V_n\to V_{n+1}$ for $n\geq0$ and $0\leq i\leq n$,
satisfying the simplicial identities. In particular, no face map is
specified in degree zero. The equivalence uses the standard
generators-and-relations presentation of $\Delta$
\cite[Chapter~I, Section~1]{GoerssJardine}.
The category $\mathcal{V}=\operatorname{Fun}
(\Delta^{\mathrm{op}},\mathsf{Vect}_k)$ has these simplicial vector spaces
as objects.
A morphism $f:V\to W$ is a natural transformation, equivalently a
family of $k$-linear maps $f_n:V_n\to W_n$ such that
$f_{n-1}d_i^V=d_i^Wf_n$ for $n\geq1$ and
$f_{n+1}s_i^V=s_i^Wf_n$ for $n\geq0$, for every admissible $i$.
Composition and identities are defined degreewise.
\end{definition}

\begin{definition}\label{def:simplicial-vector-space}
The category $\mathcal{V}$ of simplicial $k$-vector spaces is
symmetric monoidal with the levelwise tensor product
\[
(V\otimes_{\mathcal V}W)_n:=V_n\otimes_k W_n,
\]
and with unit the constant simplicial vector space $\underline{k}$.
The associativity, symmetry, and unit constraints are the levelwise
constraints inherited from $\mathsf{Vect}_k$.
For a simplicial set $K$, write $k[K]$ for the free simplicial
$k$-vector space on $K$: in degree $n$ it has basis $K_n$, and its
simplicial operators extend those of $K$ linearly.

It is closed: the internal Hom is the simplicial $k$-vector space
\[
\underline{\operatorname{Hom}}_{\mathcal{V}}(V,W)_n
:=
\operatorname{Hom}_{\mathcal{V}}\bigl(V\otimes_{\mathcal V}k[\Delta^n],W\bigr),
\]
whose simplicial operators are induced by precomposition along the
cosimplicial maps among the standard simplices $\Delta^n$. The tensor-Hom
adjunction
\[
\operatorname{Hom}_{\mathcal{V}}(X\otimes_{\mathcal V}Y,Z)
\;\cong\;
\operatorname{Hom}_{\mathcal{V}}
\bigl(X,\underline{\operatorname{Hom}}_{\mathcal{V}}(Y,Z)\bigr)
\]
is verified directly as follows.  The transpose of $f:X\otimes Y\to Z$
sends $x\in X_r$ to the map $Y\otimes k[\Delta^r]\to Z$ defined by
\[
y\otimes\theta\longmapsto f_s(\theta^*x\otimes y),
\qquad y\in Y_s,\quad\theta:[s]\to[r],
\]
extended linearly in $\theta$.  Simplicial compatibility follows from
that of $f$.  Conversely, evaluate $g_r(x)$ on
$y\otimes\operatorname{id}_{[r]}$ for
$g:X\to\underline{\operatorname{Hom}}_{\mathcal V}(Y,Z)$.
These constructions are natural and mutually inverse.
The vector-space structure on each displayed Hom set is pointwise; the
induced face and degeneracy operators are linear.
\end{definition}

The category $\mathcal{V}$ has all small limits and colimits, since
it is a functor category valued in the complete and cocomplete
category $\mathsf{Vect}_k$. For every small diagram $D:J\to\mathcal V$,
they are computed levelwise:
\[
(\lim_J D)_n\cong\lim_J(D_j)_n,\qquad
(\operatorname{colim}_J D)_n\cong\operatorname{colim}_J(D_j)_n.
\]

Write $\operatorname{Und}:\mathcal V\to\mathsf{sSet}$ for the functor
forgetting the vector-space structure degreewise. Thus
$\operatorname{Und}(X)_n$ is the underlying set of $X_n$, with the same
face and degeneracy maps regarded as set maps.

\begin{definition}\label{def:underlying-infinity}
Let $\mathcal{V}$ be the category of simplicial $k$-vector spaces of
Definition~\ref{def:category-svect}, with the closed symmetric monoidal
structure of Definition~\ref{def:simplicial-vector-space}.
The internal Hom objects equip $\mathcal V$ with its canonical
$\mathcal V$-enriched category structure, as follows. Within this
definition suppress the subscript on $\otimes_{\mathcal V}$.

\par\smallskip\noindent\emph{Evaluation, composition, and units.}
The evaluation map is the morphism corresponding to
$\operatorname{id}_{\underline{\operatorname{Hom}}_{\mathcal V}(V,W)}$
under the tensor--Hom adjunction:
\[
\operatorname{ev}_{V,W}:
\underline{\operatorname{Hom}}_{\mathcal V}(V,W)\otimes V\longrightarrow W,
\qquad
(\operatorname{ev}_{V,W})_r(a\otimes v)
=a_r(v\otimes\operatorname{id}_{[r]}).
\]
Here $a\in\underline{\operatorname{Hom}}_{\mathcal V}(V,W)_r$
is itself a simplicial linear map
$V\otimes k[\Delta^r]\to W$, and $v\in V_r$.

For three objects $V,W,Z$, enriched composition is the adjoint
\[
m_{V,W,Z}:
\underline{\operatorname{Hom}}_{\mathcal V}(W,Z)\otimes
\underline{\operatorname{Hom}}_{\mathcal V}(V,W)
\longrightarrow\underline{\operatorname{Hom}}_{\mathcal V}(V,Z)
\]
of the composite
\[
\begin{aligned}
\bigl(\underline{\operatorname{Hom}}_{\mathcal V}(W,Z)\otimes
\underline{\operatorname{Hom}}_{\mathcal V}(V,W)\bigr)\otimes V
&\xrightarrow{\ \cong\ }
\underline{\operatorname{Hom}}_{\mathcal V}(W,Z)\otimes
\bigl(\underline{\operatorname{Hom}}_{\mathcal V}(V,W)\otimes V\bigr)\\
&\xrightarrow{\ \operatorname{id}\otimes\operatorname{ev}_{V,W}\ }
\underline{\operatorname{Hom}}_{\mathcal V}(W,Z)\otimes W
\xrightarrow{\ \operatorname{ev}_{W,Z}\ }Z.
\end{aligned}
\]
Explicitly, for
$a\in\underline{\operatorname{Hom}}_{\mathcal V}(V,W)_r$,
$b\in\underline{\operatorname{Hom}}_{\mathcal V}(W,Z)_r$,
$v\in V_s$ and a basis simplex $\theta:[s]\to[r]$,
\[
\bigl(m_{V,W,Z}\bigr)_r(b\otimes a)_s(v\otimes\theta)
=b_s\bigl(a_s(v\otimes\theta)\otimes\theta\bigr),
\]
extended linearly in the input. Both maps use the same parameter
$\theta$, rather than two independent simplices.

The enriched identity is the adjoint of
$\underline{k}\otimes V\cong V\xrightarrow{\operatorname{id}_V}V$:
\[
u_V:\underline{k}\longrightarrow
\underline{\operatorname{Hom}}_{\mathcal V}(V,V),
\qquad
\bigl((u_V)_r(\lambda)\bigr)_s(v\otimes\theta)=\lambda v
\quad(\lambda\in k).
\]
\par\smallskip\noindent\emph{Enriched-category axioms.}
These are morphisms in $\mathcal V$, since evaluation and the adjunction
are simplicial and natural. With the canonical associativity and unit
isomorphisms understood, they satisfy
\[
\begin{aligned}
m_{V,Z,T}\circ\bigl(\operatorname{id}\otimes m_{V,W,Z}\bigr)
&=m_{V,W,T}\circ\bigl(m_{W,Z,T}\otimes\operatorname{id}\bigr),\\
m_{V,W,W}\circ(u_W\otimes\operatorname{id})
&=\operatorname{id}_{\underline{\operatorname{Hom}}_{\mathcal V}(V,W)},\\
m_{V,V,W}\circ(\operatorname{id}\otimes u_V)
&=\operatorname{id}_{\underline{\operatorname{Hom}}_{\mathcal V}(V,W)}.
\end{aligned}
\]
Indeed, the two associative composites have the same value
$c_s\bigl(b_s(a_s(v\otimes\theta)\otimes\theta)\otimes\theta\bigr)$
for $c\in\underline{\operatorname{Hom}}_{\mathcal V}(Z,T)_r$;
the unit equations follow by taking
$\lambda=1$. Thus the enriched structure maps satisfy the displayed
equations in $\mathcal V$. The underlying ordinary category
is the original one, since the adjunction gives
\[
\operatorname{Hom}_{\mathcal V}
\bigl(\underline{k},\underline{\operatorname{Hom}}_{\mathcal V}(V,W)\bigr)
\cong\operatorname{Hom}_{\mathcal V}(V,W).
\]

\par\smallskip\noindent\emph{Mapping spaces and the underlying $\infty$-category.}
For $V,W\in\mathcal{V}$, let
\[
\operatorname{Map}(V,W)
:=
\operatorname{Und}\bigl(\underline{\operatorname{Hom}}_{\mathcal{V}}(V,W)\bigr)
\]
be the underlying simplicial set of the internal Hom. The maps
$(b,a)\mapsto(m_{V,W,Z})_r(b\otimes a)$ and the elements
$(u_V)_r(1)$ give its simplicial composition and identities.
The underlying
$\infty$-category of $\mathcal{V}$ is
\[
\mathcal{V}^{\infty}
:=
N^{\mathrm{hc}}\bigl(\mathcal{V},\operatorname{Map}\bigr).
\]
Propositions~\ref{prop:model-structure-svect},
\ref{prop:strong-cofibrant}, and \ref{prop:simplicial-vector-kan} show that
every object is bifibrant and all mapping spaces are Kan. Hence the
homotopy coherent nerve is an $\infty$-category
\cite[Theorem~2.4.5.1, Tag~\href{https://kerodon.net/tag/00LJ}{00LJ}]{Kerodon}. If $W$ is the class of maps whose
normalized complexes are quasi-isomorphisms, then
\[
N(\mathcal V_{\mathrm{ord}})[W^{-1}]
\simeq \mathcal V^\infty
\]
by \cite[Theorem~1.3.4.20]{HA}, since its cofibrant and bifibrant
subcategories both comprise all objects of $\mathcal V$.
Here and throughout, the symbol $\mathcal{V}$ denotes both the ordinary
category of simplicial $k$-vector spaces and its self-enrichment via
$\underline{\operatorname{Hom}}$; the intended meaning is clear from
context.
\end{definition}

\smallskip
\noindent\textbf{Notation.} We use homological grading, so differentials
lower degree by one. Write $\mathsf{Ch}_{\ge0}(k)$ for the category of
chain complexes of $k$-vector spaces concentrated in nonnegative degrees
(with $d_0:C_0\to C_{-1}=0$), and $\mathsf{Ch}(k)$ for the category of
unbounded chain complexes. Morphisms in both categories are degreewise
linear maps commuting with the differentials. For
$C,D\in\mathsf{Ch}(k)$, $\operatorname{Hom}_{\mathsf{Ch}(k)}(C,D)$
denotes the set of chain maps $C\to D$.

\begin{proposition}\label{prop:model-structure-svect}
Let
\[
N:\mathcal{V}\simeq\mathsf{Ch}_{\ge0}(k)
\]
be the normalized Moore complex functor. Then $\mathcal{V}$ admits
a model structure, transferred from the projective model structure on
$\mathsf{Ch}_{\ge0}(k)$ along $N$, such that a map $f:V\to W$ in
$\mathcal{V}$ is
\begin{enumerate}
\item a weak equivalence if and only if $N(f)$ is a quasi-isomorphism in
$\mathsf{Ch}_{\ge0}(k)$;
\item a fibration if and only if
$N(f)_n:N(V)_n\to N(W)_n$ is surjective for every $n\ge1$;
\item a cofibration if and only if
$N(f)_n:N(V)_n\to N(W)_n$ is injective for every $n\ge0$.
\end{enumerate}
Moreover, every object of $\mathcal{V}$ is fibrant, and this model
structure is combinatorial and simplicial.
\end{proposition}

These two properties are used in Section~\ref{sec:function-model} to
construct the projective model structure on strict representations and
to compare it with the homotopy-coherent functor category.

\begin{proof}

For $V\in\mathcal V$, let $C_n(V)=V_n$ with differential
$\partial_n=\sum_{i=0}^n(-1)^i d_i$ when $n\geq1$, and let
$\partial_0:V_0\to0$ be zero. Quotient $C(V)$ by its degenerate
subcomplex, with $D_0(V)=0$ and
$D_n(V)=\sum_{i=0}^{n-1}s_i(V_{n-1})$ for $n\geq1$, to obtain $N(V)$.
This normalized Moore complex gives an equivalence
$N:\mathcal V\simeq\mathsf{Ch}_{\ge0}(k)$ by
\cite[Theorem~2.5.6.1, Tag~\href{https://kerodon.net/tag/00QR}{00QR}]{Kerodon} and
\cite[Remark~2.5.6.2, Tag~\href{https://kerodon.net/tag/00QS}{00QS}]{Kerodon}.

The nonnegative projective model structure on
$\mathsf{Ch}_{\ge0}(k)$ is described in
\cite[Section~5, ``Chain complexes'', p.~503]{SchwedeShipley}.
Its weak equivalences
are quasi-isomorphisms, its fibrations are surjections in positive
degrees, and its cofibrations are degreewise monomorphisms with
degreewise projective cokernel. Since $k$ is a field, every
vector space is projective, so all degreewise monomorphisms are
cofibrations. Every object is fibrant, since its map to zero is
surjective in positive degrees, and every object is cofibrant,
since its map from zero is a degreewise monomorphism.
For later reference, the standard generating sets can be written
explicitly as follows. For $m\ge0$ let $S^m$ denote the complex with $k$
in degree $m$ and zero elsewhere, and for $m\ge1$ let $D^m$ denote the
complex with $k$ in degrees $m$ and $m-1$ and differential
$d_m=\mathrm{id}$; the symbol $0$ denotes the zero complex.
More generally, for a $k$-vector space $M$ and $m\ge0$,
write $S^m M$ for the complex with $M$ in degree $m$ and zero elsewhere,
and for $m\ge1$ write $D^m M$ for the complex with $M$ in degrees $m$
and $m-1$ and identity differential; thus $S^m=S^m k$ and
$D^m=D^m k$.
For $n\geq1$, the map $S^{n-1}\to D^n$ includes the generator in
degree $n-1$. Then
\[
I_{\ge0}:=\{0\to S^0\}\cup\{S^{n-1}\to D^n:n\ge1\},\qquad
J_{\ge0}:=\{0\to D^n:n\ge1\}
\]
are generating cofibrations and generating trivial cofibrations,
respectively: fibrations are exactly the $J_{\ge0}$-injective maps and
trivial fibrations are exactly the $I_{\ge0}$-injective maps
(see the discussion following Corollary~III.2.12 on p.~157 of
\cite{GoerssJardine}, with the free abelian groups replaced by free
$k$-modules). Indeed, lifting against $0\to D^r$ specifies a
preimage of an arbitrary degree-$r$ element; lifting against
$S^{r-1}\to D^r$ specifies a preimage with prescribed boundary.
Together with $0\to S^0$, these give exactly the stated criteria.

Let $W_{\mathsf{Ch}}$, $\mathcal{F}_{\mathsf{Ch}}$, and
$\mathcal{C}_{\mathsf{Ch}}$ denote the classes of weak equivalences,
fibrations, and cofibrations of this model structure on
$\mathsf{Ch}_{\ge0}(k)$. Define a map $f$ in $\mathcal{V}$ to be a
weak equivalence, fibration, or cofibration according as $N(f)$ lies in
$W_{\mathsf{Ch}}$, $\mathcal{F}_{\mathsf{Ch}}$, or
$\mathcal{C}_{\mathsf{Ch}}$. Because $N$ is an equivalence of categories,
these classes satisfy the model category axioms: lifting, factorization,
retracts, and two-out-of-three all transfer along $N$ and a chosen
quasi-inverse. This proves the three characterizations in the
proposition.

The sets $I_{\ge0}$ and $J_{\ge0}$ transfer along a quasi-inverse of
$N$ and generate the transported structure. Moreover,
$\mathsf{Ch}_{\ge0}(k)$ is locally finitely presentable: each complex
is the filtered union of its finite-dimensional subcomplexes supported
in finitely many degrees. Equivalence of categories transports local
presentability, so the model structure on $\mathcal V$ is
combinatorial.
For a simplicial set $K$, the standard simplicial copower is
$V\otimes K:=V\otimes_{\mathcal V}k[K]$; in degree $m$ it is
$\bigoplus_{\sigma\in K_m}V_m$. Its mapping simplicial set is
$\operatorname{Und}\underline{\operatorname{Hom}}_{\mathcal V}(V,W)$.
The standard model structure of
\cite[Chapter~II, Theorem~4.4, p.~100, and Example~5.2(1), p.~103]{GoerssJardine}
has weak equivalences detected by normalized homology
\cite[Chapter~III, Corollary~2.7, p.~153]{GoerssJardine} and fibrations detected
by surjectivity of $N(f)_n$ for $n\geq1$
\cite[Chapter~III, Lemma~2.11(2), pp.~155--156]{GoerssJardine}.
These last two criteria apply to simplicial $k$-modules by forgetting
to simplicial abelian groups; they do not identify the respective
classes of cofibrations. Thus it agrees with
the transported model structure, and its simplicial model axiom applies
to this copower and these mapping spaces.

\end{proof}

\begin{proposition}\label{prop:strong-cofibrant}
If $k$ is a field, every object of $\mathsf{Ch}_{\ge0}(k)$ is
cofibrant with respect to the nonnegative projective model structure of
Proposition~\ref{prop:model-structure-svect}.
Consequently every object of $\mathcal V$ is cofibrant, and hence
bifibrant.
\end{proposition}

\begin{proof}
The cofibration criterion proved in
Proposition~\ref{prop:model-structure-svect} applies to $0\to C$, so every
complex is cofibrant. A concrete decomposition over the field $k$ is
obtained by choosing complements for
$B_n\subseteq Z_n\subseteq C_n$, where
$Z_n=\ker d_n$ and $B_n=\operatorname{im}d_{n+1}$. The differential
identifies a complement of $Z_n$ in $C_n$ with $B_{n-1}$.
These noncanonical choices decompose $C$ as
\[
C
\;\cong\;
\bigoplus_n S^n H_n(C)
\;\oplus\;
\bigoplus_n D^{n+1} B_n,
\]
where $H_n=Z_n/B_n$ and $D^{n+1}B_n$ denotes the disk complex with
$B_n$ in degrees
$n+1,n$ and identity differential. These sphere and disk summands are
cofibrant by the same degreewise criterion.
Combined with Proposition~\ref{prop:model-structure-svect}, every object
is both fibrant and cofibrant over a field.
\end{proof}

\begin{proposition}\label{prop:simplicial-vector-kan}
The underlying simplicial set of any simplicial vector space over $k$
is a Kan complex. Moreover, let
$U\subseteq\underline{\operatorname{Hom}}_{\mathcal{V}}(V,W)$ be a
subsimplicial set such that, for every $n\ge0$, $U_n$ is a linear
subspace of
$\underline{\operatorname{Hom}}_{\mathcal{V}}(V,W)_n$, and the face and
degeneracy maps of the ambient simplicial vector space restrict to $U$.
Then $U$ is a simplicial vector space and its underlying simplicial set
is Kan.
\end{proposition}

\begin{proof}

For any simplicial vector space $X$, forgetting scalar multiplication
leaves a simplicial abelian group. Its underlying simplicial set is Kan
by Proposition~\ref{prop:simplicial-group-kan}, proving the first
assertion.

For the second assertion, the ambient face and degeneracy maps restrict
to linear maps between the subspaces $U_n$. Their simplicial identities
are inherited from the ambient internal Hom, so $U$ is itself a
simplicial vector space. Applying the first assertion to $U$ shows
that its underlying simplicial set is Kan.

\end{proof}

By Proposition~\ref{prop:simplicial-vector-kan}, the mapping space
$\operatorname{Map}(V,W)$ of
Definition~\ref{def:underlying-infinity} is a Kan complex for all
$V,W\in\mathcal{V}$.

\subsection{Dold--Kan and internal Hom}\label{sec:dold-kan-internal-hom}

For $C,D\in\mathsf{Ch}(k)$, define the internal Hom complex by
\[
\underline{\operatorname{Hom}}_{\mathsf{Ch}(k)}(C,D)_i
:=
\prod_{p\in\mathbb Z}\operatorname{Hom}_k(C_p,D_{p+i}),
\qquad
(df)_p:=d_Df_p-(-1)^if_{p-1}d_C.
\]
The product permits infinitely many nonzero components in a fixed
degree. The complex can have negative degrees even for nonnegative
$C,D$: for example,
$\underline{\operatorname{Hom}}_{\mathsf{Ch}(k)}
(k[1],k[0])_{-1}\cong k$.

\begin{lemma}[the nonnegative chain complex category is closed]
\label{lem:nonnegative-chain-complexes-closed}
For a chain complex $C$, write $\tau_{\ge0}C$ for the good
truncation with $(\tau_{\ge0}C)_n=C_n$ for $n\ge1$,
$(\tau_{\ge0}C)_0=\ker(d_0)$, and $(\tau_{\ge0}C)_n=0$ for $n<0$.
For $C,D\in\mathsf{Ch}_{\ge0}(k)$, define
\[
\underline{\operatorname{Hom}}_{\mathsf{Ch}_{\ge0}(k)}(C,D)
:=
\tau_{\ge0}\bigl(\underline{\operatorname{Hom}}_{\mathsf{Ch}(k)}(C,D)\bigr).
\]
Then
$(\mathsf{Ch}_{\ge0}(k),\otimes_k,k[0],
\underline{\operatorname{Hom}}_{\mathsf{Ch}_{\ge0}(k)})$ is a closed
symmetric monoidal category. This is also recorded in
\cite[Lemma 1.13(1)]{Ngompe2024}.
\end{lemma}

{
\begin{proof}
The total tensor product has
$(X\otimes_kY)_n=\bigoplus_{i+j=n}X_i\otimes_kY_j$ and
$d(x\otimes y)=d_Xx\otimes y+(-1)^i x\otimes d_Yy$ for $x\in X_i$.
Its symmetry sends $x\otimes y$ to $(-1)^{ij}y\otimes x$
for $y\in Y_j$, with unit $k[0]$. The usual sign check makes these
chain maps and gives the symmetric monoidal axioms. We claim first that, for
$X\in\mathsf{Ch}_{\ge0}(k)$ and $A\in\mathsf{Ch}(k)$,
\[
\operatorname{Hom}_{\mathsf{Ch}(k)}(X,A)
\;\cong\;
\operatorname{Hom}_{\mathsf{Ch}_{\ge0}(k)}\bigl(X,\tau_{\ge0}A\bigr).
\]
Indeed, a chain map $f:X\to A$ has $d_0\circ f_0=f_{-1}\circ d_X=0$,
so $f_0$ lands in $\ker(d_0)$; for $n\ge1$, $f_n$ lands in $A_n$. Thus
$f$ factors uniquely through $\tau_{\ge0}A$, and the converse is
composition with the inclusion. Next, the unbounded tensor-Hom adjunction
\[
\operatorname{Hom}_{\mathsf{Ch}(k)}(X\otimes_k Y,Z)
\;\cong\;
\operatorname{Hom}_{\mathsf{Ch}(k)}
\bigl(X,\underline{\operatorname{Hom}}_{\mathsf{Ch}(k)}(Y,Z)\bigr)
\]
is obtained degreewise from
$\operatorname{Hom}_k(X_i\otimes_k Y_j,Z_{i+j})
\cong\operatorname{Hom}_k(X_i,\operatorname{Hom}_k(Y_j,Z_{i+j}))$.
The tensor and internal-Hom differential formulas above identify the
chain-map equations on the two sides; the product in the internal Hom
allows all degreewise components of the transpose.
Combining the two isomorphisms gives, for
$X,Y,Z\in\mathsf{Ch}_{\ge0}(k)$,
\[
\begin{aligned}
\operatorname{Hom}_{\mathsf{Ch}_{\ge0}(k)}(X\otimes_k Y,Z)
&\cong
\operatorname{Hom}_{\mathsf{Ch}(k)}(X\otimes_k Y,Z)\\
&\cong
\operatorname{Hom}_{\mathsf{Ch}(k)}
\bigl(X,\underline{\operatorname{Hom}}_{\mathsf{Ch}(k)}(Y,Z)\bigr)\\
&\cong
\operatorname{Hom}_{\mathsf{Ch}_{\ge0}(k)}
\bigl(X,\underline{\operatorname{Hom}}_{\mathsf{Ch}_{\ge0}(k)}(Y,Z)\bigr).
\end{aligned}
\]
This is natural in $X,Y,Z$, so the internal Hom adjunction holds.
\end{proof}
}

\begin{proposition}[enriched Dold--Kan comparison]
\label{prop:enriched-dold-kan-comparison}
For $V,W\in\mathcal V$, evaluation and the Eilenberg--Zilber shuffle
define a natural chain map
\[
\Theta_{V,W}:
N\underline{\operatorname{Hom}}_{\mathcal V}(V,W)
\longrightarrow
\tau_{\ge0}\underline{\operatorname{Hom}}_{\mathsf{Ch}(k)}(NV,NW).
\]
It is a quasi-isomorphism. Since $k$ is a field, each component
$\Theta_{V,W}$ is a chain-homotopy equivalence.
\end{proposition}

\begin{proof}
Put $M=\underline{\operatorname{Hom}}_{\mathcal V}(V,W)$,
$C=NV$, and $D=NW$. We use the normalized $k$-linear shuffle
\[
\nabla_{A,B}:NA\otimes_k NB\longrightarrow
N(A\otimes_{\mathcal V}B).
\]
Indeed, the unnormalized shuffle for simplicial $k$-modules
\cite[Theorem~8.5.1, Application~8.5.3, and Section~8.5.4,
pp.~276--278]{Weibel} carries a degenerate input
to a degenerate output, by the simplicial identities
\cite[Proposition~2.5.7.8, Tag~\href{https://kerodon.net/tag/00RP}{00RP}]{Kerodon}. The same calculation applies over $k$,
since it uses only simplicial identities and $k$-balanced sums of tensors.
It therefore descends to
degeneracy quotients. The quotient maps are quasi-isomorphisms
\cite[Lemma~8.3.7 and Theorem~8.3.8, pp.~266--267]{Weibel},
as is the unnormalized shuffle; over a field, tensoring preserves
quasi-isomorphisms. Thus $\nabla_{A,B}$ is a natural quasi-isomorphism
for the levelwise tensor product used here.
The evaluation map and this shuffle give
\[
NM\otimes_k C
\xrightarrow{\nabla}N(M\otimes_{\mathcal V}V)
\xrightarrow{N(\operatorname{ev})}D.
\]
By the chain tensor--Hom adjunction this composite is adjoint to a map
$NM\to\underline{\operatorname{Hom}}_{\mathsf{Ch}(k)}(C,D)$. Since
$NM$ is nonnegative, the map factors uniquely through the good
truncation of the target; this is $\Theta_{V,W}$. Naturality follows from
naturality of evaluation and shuffle.

We check homology. For every $i\ge0$, put
$E_i:=k[\Delta^i]/k[\partial\Delta^i]$. Its normalized complex is
$NE_i\cong S^ik$; in particular, $E_0=k$ since $\partial\Delta^0$ is empty.
Thus $E_i$ models the reduced free vector space on the pointed
$i$-sphere; for $i=0$, the non-basepoint contributes the one copy of $k$.
Since all objects in the two Dold--Kan
model categories are bifibrant, the tensor--Hom adjunction, Dold--Kan,
and the shuffle equivalence give
\[
\begin{aligned}
H_i(NM)
&\cong \pi_i(M,0)\\
&\cong
\operatorname{Hom}_{\operatorname{Ho}(\mathcal V)}
(E_i\otimes_{\mathcal V}V,W)\\
&\cong
\operatorname{Hom}_{\operatorname{Ho}(\mathsf{Ch}_{\ge0}(k))}
(S^ik\otimes_k C,D).
\end{aligned}
\]
Here the first isomorphism is Moore's theorem
(\cite[Chapter~III, Corollary~2.7, p.~153]{GoerssJardine}), and the last uses the normalized
$k$-linear shuffle above, together with Dold--Kan.

A chain map $F:S^ik\otimes C\to D$ is the same as a family
$f_p:C_p\to D_{p+i}$. The chain-map equation is
$d_Df_p-(-1)^if_{p-1}d_C=0$, exactly the assertion that $f$ is an
$i$-cycle in $\underline{\operatorname{Hom}}_{\mathsf{Ch}(k)}(C,D)$.
A chain homotopy between the maps corresponding to $f$ and $g$ is a
family $h_p:C_p\to D_{p+i+1}$ satisfying $f-g=dh$. Thus
\[
\operatorname{Hom}_{\operatorname{Ho}(\mathsf{Ch}_{\ge0}(k))}
(S^ik\otimes C,D)
\cong H_i\underline{\operatorname{Hom}}_{\mathsf{Ch}(k)}(C,D).
\]
To identify this isomorphism with $H_i(\Theta_{V,W})$, represent a
cycle by $\alpha:E_i\to M$ and let
$F=\operatorname{ev}\circ(\alpha\otimes\operatorname{id}_V)$.
The naturality of the normalized shuffle gives the chain-map identity
\[
N(F)\circ\nabla_{E_i,V}
=N(\operatorname{ev})\circ\nabla_{M,V}
\circ(N\alpha\otimes\operatorname{id}_C).
\]
The left side represents the preceding homotopy-class identification,
while the transpose of the right side represents $\Theta_{V,W}$ on
the chosen cycle. Thus $H_i(\Theta_{V,W})$ is the displayed
isomorphism. Good truncation does not change homology in degrees
$i\ge0$, so $\Theta_{V,W}$ is a quasi-isomorphism.

Finally, an acyclic complex of vector spaces is contractible because all
short exact sequences split. The cone of each $\Theta_{V,W}$ is therefore
contractible, which proves the last assertion.
\end{proof}

\begin{remark}\label{rem:dold-kan-monoidal}

The pointwise tensor product on $\mathcal{V}$ is given levelwise by
\[
(V\otimes_{\mathcal V}W)_n:=V_n\otimes_k W_n,
\]
while the usual total tensor product on $\mathsf{Ch}_{\ge0}(k)$ is given
by
\[
(C\otimes_k D)_n:=\bigoplus_{p+q=n}C_p\otimes_k D_q,
\qquad
d(x\otimes_k y):=dx\otimes_k y+(-1)^{|x|}x\otimes_k dy.
\]
These structures are related by natural chain maps
\[
NV\otimes_k NW
\xrightarrow{\ \nabla\ }N(V\otimes_{\mathcal V}W)
\xrightarrow{\ \mathrm{AW}\ }NV\otimes_k NW,
\qquad \mathrm{AW}\circ\nabla=\operatorname{id}.
\]
The shuffle $\nabla$ is a quasi-isomorphism, as established in
Proposition~\ref{prop:enriched-dold-kan-comparison}. The
Alexander--Whitney chain map is given by
\cite[Construction~2.5.8.2, Tag~\href{https://kerodon.net/tag/00S2}{00S2}]{Kerodon} and
\cite[Proposition~2.5.8.3, Tag~\href{https://kerodon.net/tag/00S3}{00S3}]{Kerodon}.
Applied to a degenerate simplex, each front--back summand is degenerate
in at least one tensor factor
\cite[Proposition~2.5.8.5, Tag~\href{https://kerodon.net/tag/00S5}{00S5}]{Kerodon}, so it descends to the normalized map of
\cite[Construction~2.5.8.6, Tag~\href{https://kerodon.net/tag/00S6}{00S6}]{Kerodon}.
The identity $\mathrm{AW}\circ\nabla=\operatorname{id}$ is
\cite[Proposition~2.5.8.9, Tag~\href{https://kerodon.net/tag/00S9}{00S9}]{Kerodon}.
These formulas are $k$-balanced, and the chain-map, degeneracy and
composition identities use only simplicial identities; thus the cited
results for simplicial abelian groups apply with tensor products over $k$.
Hence both maps are chain-homotopy equivalences over $k$.

For $V=W=k[\Delta^1]$, the degree-one spaces of $NV\otimes_k NW$ and
$N(V\otimes_{\mathcal V}W)$ have dimensions $4$ and $5$,
respectively: $\Delta^1\times\Delta^1$ has five nondegenerate edges.
Thus the comparison is a chain-homotopy equivalence rather than
a strict isomorphism in general.
Together with the internal Hom comparison of
Section~\ref{sec:dold-kan-internal-hom}, these natural chain
equivalences are the tensor comparisons used below.

\end{remark}

\subsection{Mapping spaces and their homotopy groups}\label{sec:mapping-homotopy}

\begin{theorem}\label{thm:mapping-homology-formula}
For $V,W\in\mathcal V$ and $i\ge0$, with homotopy groups based at the
zero map,
\[
\pi_i\operatorname{Map}(V,W)
\cong
H_i\underline{\operatorname{Hom}}_{\mathsf{Ch}(k)}(NV,NW).
\]
Moreover, for all $C,D\in\mathsf{Ch}_{\ge0}(k)$,
\[
H_i\underline{\operatorname{Hom}}_{\mathsf{Ch}(k)}(C,D)
\cong
\prod_{p\ge0}
\operatorname{Hom}_k\bigl(H_p(C),H_{p+i}(D)\bigr).
\]
For $i=0$, $\pi_0$ is the vector space of connected components of the
simplicial abelian group $\operatorname{Map}(V,W)$. The product is not
assumed finite.
\end{theorem}

\begin{proof}
By Proposition~\ref{prop:enriched-dold-kan-comparison},
\[
N\underline{\operatorname{Hom}}_{\mathcal V}(V,W)
\simeq
\tau_{\ge0}\underline{\operatorname{Hom}}_{\mathsf{Ch}(k)}(NV,NW).
\]
The underlying simplicial set of the left-hand simplicial vector space is
$\operatorname{Map}(V,W)$. The Moore comparison
\cite[Chapter~III, Corollary~2.7, p.~153]{GoerssJardine} identifies its based
homotopy groups with the homology of its normalization, including the
additive group of components in degree zero. Good truncation preserves
homology in every degree $i\ge0$, proving the first formula.

For the second formula, use the noncanonical decomposition from the proof
of Proposition~\ref{prop:strong-cofibrant}, with
$B_p=\operatorname{im}d_{p+1}$:
\[
C\cong
\bigoplus_{p\ge0}S^pH_p(C)
\oplus
\bigoplus_{p\ge0}D^{p+1}B_p.
\]
Every disk summand is contractible. Precomposition with a contracting
homotopy contracts its internal Hom into $D$, so the disk summands do not
contribute homology. Since internal Hom changes direct sums in its first
variable into products, it remains to compute
\[
\prod_{p\ge0}
\underline{\operatorname{Hom}}_{\mathsf{Ch}(k)}(S^pH_p(C),D).
\]
The $i$-th homology of the $p$-th factor is
$\operatorname{Hom}_k(H_p(C),H_{p+i}(D))$, because
$\operatorname{Hom}_k(H_p(C),-)$ is exact. Products of vector spaces are
exact, giving the displayed product formula.

The resulting map is the canonical map induced by homology. The chosen
splittings prove that it is an isomorphism, while the map itself is
independent of those choices. The product cannot in general be replaced
by a direct sum; the two agree when only finitely many factors are
nonzero.
\end{proof}

\begin{corollary}[high mapping-space homotopy vanishes]
\label{cor:mapping-homotopy-vanishing}
Let $n\geq0$. If the homology of $C,D\in\mathsf{Ch}_{\ge0}(k)$ is
concentrated in $[0,n]$, with no boundedness assumption on the
complexes themselves, then
\[
H_i\underline{\operatorname{Hom}}_{\mathsf{Ch}(k)}(C,D)=0
\qquad(i>n).
\]
Consequently, if $V,W\in\mathcal V$ have homotopy groups concentrated in
$[0,n]$, then $\pi_i(\operatorname{Map}(V,W),0)=0$ for $i>n$.
\end{corollary}

\begin{proof}
In the product formula of Theorem~\ref{thm:mapping-homology-formula}, if $i>n$ and
$p\ge0$, then $p+i>n$, so every factor $H_{p+i}(D)$ vanishes. The last
claim follows from the first formula of the theorem and Dold--Kan.
\end{proof}

\subsection{Truncation and homotopy $n$-categories}\label{sec:truncation}

\paragraph{Strict models and local truncation.}
In strict universal properties, an $(n,1)$-category means the
simplicial-set model of \cite[Definition~4.7.7.9, Tag~\href{https://kerodon.net/tag/057M}{057M}]{Kerodon}.
An $\infty$-category is \emph{locally $(n-1)$-truncated} if every
mapping space is $(n-1)$-truncated; for $n=0$, each is either empty or
contractible. By \cite[Corollary~4.7.8.3, Tag~\href{https://kerodon.net/tag/058S}{058S}]{Kerodon}, such a category
is equivalent to a strict $(n,1)$-model, although its given simplicial
set need not itself be strict. For example, the nerve of the groupoid
with two distinct objects and exactly one morphism between each ordered
pair is locally $(-1)$-truncated but
fails the vertex minimality required of a strict $0$-category.
Proposition~\ref{prop:homotopy-n-universal} supplies the equivalence
version needed for our coefficient category.

\begin{definition}\label{def:n-truncated}
For $n\geq0$, a Kan complex $X$ is $n$-truncated if, for every
$m\ge n+2$, every map $\partial\Delta^m\to X$ admits an extension
$\Delta^m\to X$ \cite[Definition~3.5.7.1, Tag~\href{https://kerodon.net/tag/054P}{054P}]{Kerodon}.
An extension is required to exist, not to be unique.
\end{definition}

\begin{definition}[homotopy $n$-truncation]\label{def:homotopy-n-truncation}
For a Kan complex $X$ and $n\ge0$, let $\pi_{\le n}(X)$ denote the
homotopy $n$-truncation of $X$ as defined in
\cite[Construction~3.5.6.10, Tag~\href{https://kerodon.net/tag/054D}{054D}]{Kerodon}.
The properties used below are the following:
\begin{enumerate}
\item $\pi_{\le n}(X)$ is a Kan complex, in fact an $n$-groupoid, that
  is, a Kan complex in which every horn of dimension $>n$ has a
  unique filler \cite[Corollary~3.5.6.16, Tag~\href{https://kerodon.net/tag/054L}{054L}]{Kerodon} and
\cite[Definition~3.5.5.1, Tag~\href{https://kerodon.net/tag/053M}{053M}]{Kerodon}. For $n=0$ it is the discrete
  simplicial set $\pi_0(X)$
  \cite[Example~3.5.6.8, Tag~\href{https://kerodon.net/tag/054B}{054B}]{Kerodon};
\item the tautological map $u_X:X\to\pi_{\le n}(X)$ exhibits
  $\pi_{\le n}(X)$ as an $n$-truncation of $X$
  \cite[Example~3.5.7.26, Tag~\href{https://kerodon.net/tag/055D}{055D}]{Kerodon}. In particular, it
  induces a bijection on $\pi_0$ and isomorphisms on based $\pi_m$ for
  $1\leq m\leq n$;
\item $u_X$ also exhibits $\pi_{\le n}(X)$ as the homotopy
  $n$-category of $X$ in the sense of
  Definition~\ref{def:homotopy-n-category}
  \cite[Example~4.7.9.6, Tag~\href{https://kerodon.net/tag/059C}{059C}]{Kerodon};
\item for every strict $(n,1)$-category $\mathcal{D}$, precomposition
  with $u_X$ has the strict functor-category universal property
  in Definition~\ref{def:homotopy-n-category}
  \cite[Proposition~4.7.9.7, Tag~\href{https://kerodon.net/tag/059D}{059D}]{Kerodon};
\item if $f:X\to Y$ is a map of Kan complexes, then $\pi_{\le n}(f)$
  is an equivalence of Kan complexes if and only if $\pi_0(f)$ is a
  bijection and, for
  every vertex $x\in X$, the map
  $\pi_m(X,x)\to\pi_m(Y,f(x))$ is an isomorphism for $1\le m\le n$.
\end{enumerate}
For the last assertion, naturality of $u_X$ and the preceding
$n$-truncation property identify the low-dimensional homotopy maps of
$\pi_{\le n}(f)$ with those of $f$. Higher homotopy groups vanish in
both target $n$-groupoids, so the criterion follows from the Whitehead
theorem \cite[Theorem~3.2.7.1, Tag~\href{https://kerodon.net/tag/00WV}{00WV}]{Kerodon}. For $n=0$, only the
$\pi_0$ condition remains.
\end{definition}

\begin{definition}[homotopy $n$-category]\label{def:homotopy-n-category}
Let $\mathcal{C}$ be an $\infty$-category and let $n\ge0$. A functor
$F:\mathcal{C}\to\mathcal{C}'$ exhibits $\mathcal{C}'$ as a homotopy
$n$-category of $\mathcal{C}$ if $\mathcal{C}'$ is an $(n,1)$-category
and, for every $(n,1)$-category $\mathcal{D}$, precomposition with $F$
induces a bijection between the sets of functors
\[
\operatorname{Hom}_{\mathrm{sSet}}(\mathcal{C}',\mathcal{D})
\xrightarrow{\ \cong\ }
\operatorname{Hom}_{\mathrm{sSet}}(\mathcal{C},\mathcal{D}).
\]
This is \cite[Definition~4.7.9.1, Tag~\href{https://kerodon.net/tag/0597}{0597}]{Kerodon}. By
\cite[Proposition~4.7.9.7, Tag~\href{https://kerodon.net/tag/059D}{059D}]{Kerodon}, it is equivalent to
requiring precomposition to induce a strict isomorphism of simplicial sets
\[
\operatorname{Fun}(\mathcal{C}',\mathcal{D})
\xrightarrow{\ \cong\ }
\operatorname{Fun}(\mathcal{C},\mathcal{D})
\]
of functor $\infty$-categories. Such an object exists for every
$\infty$-category and every $n\ge0$
\cite[Corollary~4.7.9.16, Tag~\href{https://kerodon.net/tag/059P}{059P}]{Kerodon}. The strict universal
property makes it unique up to a unique isomorphism compatible with
the comparison functor, and makes it functorial in $\mathcal C$.
We denote it by
$h_{\le n}(\mathcal{C})$.
\end{definition}

\begin{proposition}[properties of homotopy $n$-categories]
\label{prop:homotopy-n-properties}
Let $\mathcal{C}$ be an $\infty$-category and let $n\ge0$. Then:
\begin{enumerate}
\item $h_{\le n}(\mathcal{C})$ is an $(n,1)$-category. The
  construction is functorial in $\mathcal{C}$
  \cite[Corollary~4.7.9.16, Tag~\href{https://kerodon.net/tag/059P}{059P}]{Kerodon};
\item the comparison map has the strict universal property of
  Definition~\ref{def:homotopy-n-category}
  \cite[Proposition~4.7.9.7, Tag~\href{https://kerodon.net/tag/059D}{059D}]{Kerodon};
\item the comparison map exhibits $h_{\le n}(\mathcal{C})$ as a local
  $(n-1)$-truncation of $\mathcal{C}$; in particular, if
  $\mathcal{C}$ is locally $(n-1)$-truncated, then the comparison map is
  an equivalence
  \cite[Corollary~4.7.9.16(4), Tag~\href{https://kerodon.net/tag/059P}{059P}]{Kerodon};
\item if $X$ is a Kan complex, then
  $h_{\le n}(X)\simeq\pi_{\le n}(X)$
  \cite[Example~4.7.9.6, Tag~\href{https://kerodon.net/tag/059C}{059C}]{Kerodon};
\item for $n\ge1$ and $X,Y\in\mathcal{C}$, writing $F$ for the
  comparison functor, its induced map on mapping spaces is the
  $(n-1)$-truncation:
  \[
  \operatorname{Map}_{h_{\le n}(\mathcal{C})}(FX,FY)
  \;\simeq\;
  \tau_{\le n-1}\bigl(\operatorname{Map}_{\mathcal{C}}(X,Y)\bigr)
  \]
  \cite[Corollary~4.7.6.24, Tag~\href{https://kerodon.net/tag/06RT}{06RT}]{Kerodon}.
\end{enumerate}
\end{proposition}

{
\begin{proof}
Assertions (1)--(4) follow from the cited universal properties and
local-truncation result. For (5), apply the mapping-space criterion
for a local $(n-1)$-truncation to the comparison functor
\cite[Corollary~4.7.6.24, Tag~\href{https://kerodon.net/tag/06RT}{06RT}]{Kerodon}.
\end{proof}
}

\begin{definition}[$n$-linear spaces]\label{def:n-linear-space}
\label{def:truncated-infinity-category}
Let $n\ge1$. An \emph{$n$-linear space} is an object of
$\mathcal V^{\infty}$ represented by a simplicial vector space $V$
satisfying
\[
H_i(NV)=0\qquad(i\ge n).
\]
The condition is invariant under weak equivalence, since weak
equivalences in $\mathcal V$ become quasi-isomorphisms under $N$.
Write $\mathbf{Vect}^{(n)}_k\subseteq\mathcal V^{\infty}$ for the full
$\infty$-subcategory of these objects. Its mapping spaces are those of
$\mathcal V^{\infty}$ between the indicated objects.
It is presented by the coherent nerve of the full simplicial
subcategory of $(\mathcal V,\operatorname{Map})$ on these representatives.

By Dold--Kan and Moore's theorem,
$\pi_i(\operatorname{Und}V,0)\cong H_i(NV)$.
Translation by any vertex is an automorphism of the underlying
simplicial set, so the same vanishing holds at every basepoint.
Thus the condition is equivalent to $\operatorname{Und}V$ being
$(n-1)$-truncated \cite[Proposition~3.5.7.9, Tag~\href{https://kerodon.net/tag/054V}{054V}]{Kerodon}.
\end{definition}

Every $n$-linear space has a chain-complex representative concentrated in
degrees $0,\ldots,n-1$. Indeed, the good upper truncation
$\tau_{\le n-1}NV$ agrees with $NV$ below degree $n-1$, has
$(NV)_{n-1}/\operatorname{im}d_n$ in degree $n-1$, and vanishes above.
The natural map $NV\to\tau_{\le n-1}NV$ is a quasi-isomorphism; apply
Dold--Kan to obtain the corresponding weakly equivalent simplicial
vector space. The original normalized complex may still have
contractible terms in higher degrees.

\begin{proposition}[categorical height of $n$-linear spaces]
\label{prop:n-linear-category-height}
For $n\ge1$, the mapping spaces of $\mathbf{Vect}^{(n)}_k$ are
$(n-1)$-truncated.  Consequently $\mathbf{Vect}^{(n)}_k$ is locally
$(n-1)$-truncated, and its comparison with a strict $(n,1)$-model is an equivalence
\[
\mathbf{Vect}^{(n)}_k
\xrightarrow{\ \simeq\ }
h_{\le n}\bigl(\mathbf{Vect}^{(n)}_k\bigr).
\]
For $n=1$ there is an equivalence
\[
\mathbf{Vect}^{(1)}_k\simeq N(\mathsf{Vect}_k).
\]
\end{proposition}

\begin{proof}
For $V,W\in\mathbf{Vect}^{(n)}_k$, Theorem~
\ref{thm:mapping-homology-formula} gives
\[
\pi_i\operatorname{Map}(V,W)
\cong
\prod_{p\ge0}\operatorname{Hom}_k
\bigl(H_p(NV),H_{p+i}(NW)\bigr).
\]
If $i\ge n$, then $p+i\ge n$ for every $p\ge0$, so the right-hand side
vanishes. Translation by a vertex of the simplicial abelian mapping
space gives the same vanishing at every basepoint. Thus each mapping
space is $(n-1)$-truncated, and
Proposition~\ref{prop:homotopy-n-properties}(3) gives the comparison
equivalence. When $n=1$, every object is equivalent to a degree-zero
constant simplicial vector space. Between two such constant objects
the mapping space is strictly the constant set of linear maps, since
each $\Delta^m$ is connected. Their full subcategory is therefore
the nerve of $\mathsf{Vect}_k$ and is essentially surjective, proving
the last equivalence.
The original mapping spaces need not be literally discrete: for
$V=N^{-1}(k\xrightarrow{\mathrm{id}}k)$ in degrees $1,0$,
the normalized complex is acyclic, so the mapping-space formula makes
$\operatorname{Map}(V,V)$ contractible. Yet $V_0=k\ne0$, and the zero
map and identity act differently there; they are distinct vertices
of this contractible mapping space.
\end{proof}

\begin{construction}[the two-term model]\label{constr:two-term-linear-space}
For $n=2$, a chain-complex representative is a two-term complex
$\partial:V_1\to V_0$.  It determines an internal category in vector spaces
whose object space is $V_0$, whose arrow space is $V_0\oplus V_1$, and whose
structure maps are
\[
s(x,u)=x,
\qquad t(x,u)=x+\partial u,
\qquad e(x)=(x,0),
\]
\[
(x+\partial u,v)\circ(x,u)=(x,u+v).
\]
The identities and associativity follow from addition; the inverse
arrow to $(x,u)$ is $(x+\partial u,-u)$, so this is an internal
groupoid. Conversely, an internal category $A$ in vector spaces gives
the complex
$\ker(s:A_1\to A_0)\xrightarrow{\,t\,}A_0$.
Splitting an arrow as
$a=e(s(a))+(a-e(s(a)))$ identifies the two constructions up to natural
isomorphism \cite[Proposition~8]{BaezCrans2004}. Linear natural
transformations correspond to chain homotopies, giving the
$2$-categorical equivalence
\cite[Theorem~12]{BaezCrans2004}.
\end{construction}

\begin{proposition}[universal property of the homotopy
$n$-category]\label{prop:homotopy-n-universal}
Let $n\ge0$, let $X$ be a Kan complex, and let $\mathcal D$ be a locally
$(n-1)$-truncated $\infty$-category.  The natural map
\[
u_X^*:
\operatorname{Fun}\bigl(\pi_{\le n}(X),\mathcal{D}\bigr)
\xrightarrow{\ \simeq\ }
\operatorname{Fun}(X,\mathcal{D})
\]
is an equivalence of $\infty$-categories. If $\mathcal D$ is itself a
strict $(n,1)$-model, this map is a strict isomorphism of functor
simplicial sets.
\end{proposition}

\begin{proof}
For a strict target this is
\cite[Example~4.7.9.6, Tag~\href{https://kerodon.net/tag/059C}{059C}]{Kerodon} and
\cite[Proposition~4.7.9.7, Tag~\href{https://kerodon.net/tag/059D}{059D}]{Kerodon}.
In general, Proposition~\ref{prop:homotopy-n-properties}(3) gives an
equivalence $e:\mathcal D\to h_{\le n}\mathcal D$.
Postcomposition with $e$ gives the vertical equivalences in
\[
\begin{CD}
\operatorname{Fun}(\pi_{\le n}X,\mathcal D) @>{u_X^*}>>
\operatorname{Fun}(X,\mathcal D)\\
@V{e_*}V{\simeq}V @V{e_*}V{\simeq}V\\
\operatorname{Fun}(\pi_{\le n}X,h_{\le n}\mathcal D) @>{u_X^*}>>
\operatorname{Fun}(X,h_{\le n}\mathcal D).
\end{CD}
\]
The bottom functor is an isomorphism, so the top functor is an equivalence.
Here postcomposition preserves equivalences by carrying a quasi-inverse
and its equivalence homotopies to the corresponding functor categories.
\end{proof}

\begin{lemma}[precomposition by an equivalence]\label{lem:precomp-equivalence}
Let $g:A\to B$ be a homotopy equivalence of Kan complexes. Then, for
every $\infty$-category $\mathcal{D}$, precomposition
\[
g^*:
\operatorname{Fun}(B,\mathcal{D})
\longrightarrow
\operatorname{Fun}(A,\mathcal{D})
\]
is an equivalence of $\infty$-categories.
\end{lemma}

\begin{proof}
By \cite[Example~4.5.1.13, Tag~\href{https://kerodon.net/tag/01E2}{01E2}]{Kerodon}, $g$ is an
equivalence of $\infty$-categories. More explicitly, choose a
simplicial homotopy inverse $h:B\to A$. The homotopies
$hg\simeq\operatorname{id}_A$ and $gh\simeq\operatorname{id}_B$
induce natural equivalences between the composites of $g^*$ and
$h^*$ and the respective identity functors. Thus $h^*$ is a
quasi-inverse to $g^*$.
\end{proof}

\subsection{Homotopy coherent nerve and local Kan simplicial categories}\label{sec:hc-nerve}

\begin{definition}\label{def:locally-kan}
A simplicial category $\mathcal{C}$ is locally Kan if each mapping
space is a Kan complex.

A simplicial functor $F:\mathcal C\to\mathcal D$ is a
\emph{Dwyer--Kan equivalence} if
\begin{enumerate}
\item for every $X,Y\in\mathcal C$, the map
  $\operatorname{Map}_{\mathcal C}(X,Y)\to
  \operatorname{Map}_{\mathcal D}(FX,FY)$ is a weak homotopy
  equivalence; and
\item the induced functor $hF:h\mathcal C\to h\mathcal D$ is essentially
  surjective.
\end{enumerate}
Here $h\mathcal C$ is the category with the same objects as
$\mathcal C$ and morphisms
$\pi_0\operatorname{Map}_{\mathcal C}(X,Y)$
\cite[Construction~2.4.6.1, Tag~\href{https://kerodon.net/tag/00LX}{00LX}]{Kerodon}; composition descends from the
simplicial composition maps. The two conditions above are the
weakly fully faithful and weakly essentially surjective conditions
of \cite[Definition~4.8.4.8, Tag~\href{https://kerodon.net/tag/01Q0}{01Q0}]{Kerodon}. A Dwyer--Kan
equivalence is also called a weak equivalence of simplicial categories.
\end{definition}

Here and below, $\operatorname{Path}$ denotes the left
adjoint path simplicial category of the homotopy coherent nerve
\cite[Section~1.1.5, p.~23]{HTT}.

\begin{proposition}\label{prop:hc-nerve-properties}
Let $\mathcal{C}$ be a locally Kan simplicial category.
\begin{enumerate}
\item $N^{\mathrm{hc}}(\mathcal{C})$ is an $\infty$-category
  \cite[Theorem~2.4.5.1, Tag~\href{https://kerodon.net/tag/00LJ}{00LJ}]{Kerodon}.
\item For every $X,Y\in\mathcal C$, the canonical comparison map
  \[
  \operatorname{Map}_{\mathcal C}(X,Y)
  \longrightarrow
  \operatorname{Map}_{N^{\mathrm{hc}}(\mathcal C)}(X,Y)
  \]
  is a homotopy equivalence of Kan complexes
  \cite[Remark~4.6.7.6, Tag~\href{https://kerodon.net/tag/01LF}{01LF}]{Kerodon}.
\item If $\mathcal D$ is another locally Kan simplicial category and
  $F:\mathcal C\to\mathcal D$ is a simplicial functor, then
  $N^{\mathrm{hc}}(F)$ is an equivalence of $\infty$-categories if and
  only if $F$ is a Dwyer--Kan equivalence
  \cite[Corollary~4.8.4.9, Tag~\href{https://kerodon.net/tag/01Q1}{01Q1}]{Kerodon}.
\item The counit
  \[
  \operatorname{Path}\bigl[N^{\mathrm{hc}}(\mathcal{C})\bigr]
  \to\mathcal{C}
  \]
  is a Dwyer--Kan equivalence \cite[Theorem~2.2.0.1]{HTT}.
\end{enumerate}
\end{proposition}

\begin{proof}
Assertion (1) is the cited inner-horn theorem. Assertion (2) is the
cited comparison of enriched and intrinsic mapping spaces, and
assertion (3) combines that comparison with essential surjectivity.
For (4), $\operatorname{Path}[N^{\mathrm{hc}}(\mathcal C)]$ has the
same objects as $\mathcal C$, and \cite[Theorem 2.2.0.1]{HTT} says that
the counit is a weak homotopy equivalence on every mapping space. It is
therefore a Dwyer--Kan equivalence by Definition~\ref{def:locally-kan}.
\end{proof}

\begin{remark}\label{rem:kerodon-numbering}
The version convention and stable locators for Kerodon citations are
listed in Appendix A.2.
\end{remark}

\subsection{Enriched categories, bifunctors, and ends}\label{sec:enriched-ends}

\begin{definition}[$\mathbf{V}$-enriched categories, functors, and
bifunctors]\label{def:enriched-monoidal}

Let $\mathbf{V}$ be a closed symmetric monoidal category with unit $I$.
\begin{enumerate}
\item A $\mathbf{V}$-enriched category $\mathcal{C}$ consists of objects,
  a $\mathbf{V}$-object
  $\underline{\operatorname{Hom}}_{\mathcal{C}}(x,y)$ for each pair
  $x,y$, composition maps
  \[
  \underline{\operatorname{Hom}}_{\mathcal{C}}(y,z)\otimes
  \underline{\operatorname{Hom}}_{\mathcal{C}}(x,y)
  \longrightarrow
  \underline{\operatorname{Hom}}_{\mathcal{C}}(x,z),
  \]
  and identity maps
  $I\to\underline{\operatorname{Hom}}_{\mathcal{C}}(x,x)$,
  subject to the associativity and unit axioms of
  \cite[Section 1.2]{Kelly1982}.
\item A $\mathbf{V}$-enriched functor
  $F:\mathcal{C}\to\mathcal{D}$ between $\mathbf{V}$-enriched categories
  is an assignment $x\mapsto F(x)$ on objects together with
  $\mathbf{V}$-maps
  \[
  \underline{\operatorname{Hom}}_{\mathcal{C}}(x,y)
  \longrightarrow
  \underline{\operatorname{Hom}}_{\mathcal{D}}(F(x),F(y))
  \]
  compatible with composition and identities, i.e.\ satisfying the
  axioms of \cite[Section 1.2]{Kelly1982}. A $\mathbf{V}$-natural
  transformation $\alpha:F\to G$ between $\mathbf{V}$-enriched functors
  is a family
  \[
  \alpha_x:I\to
  \underline{\operatorname{Hom}}_{\mathcal{D}}(F(x),G(x))
  \]
  satisfying the $\mathbf{V}$-naturality condition of
  \cite[Section 1.2]{Kelly1982}; a $\mathbf{V}$-natural isomorphism is
  one whose every component is invertible in the underlying category
  $\mathcal D_0$, whose hom-sets are
  $\mathbf V(I,\underline{\operatorname{Hom}}_{\mathcal D}(-,-))$.
  The tensor product $\mathcal C\otimes\mathcal D$ has objects $(x,y)$
  and hom-objects
  \[
  (\mathcal C\otimes\mathcal D)((x,y),(x',y'))
  :=\underline{\operatorname{Hom}}_{\mathcal{C}}(x,x')\otimes
  \underline{\operatorname{Hom}}_{\mathcal{D}}(y,y')
  \]
  with composition using the symmetry of $\mathbf V$. A
  $\mathbf V$-enriched bifunctor is a $\mathbf V$-functor
  $T:\mathcal C\otimes\mathcal D\to\mathcal E$; equivalently, its maps
  from the displayed hom-objects to the hom-objects of $\mathcal E$
  satisfy the joint composition and unit axioms. These are the
  conventions of \cite[Section 1.4]{Kelly1982}.
\end{enumerate}
The enriched opposite $\mathcal C^{\mathrm{op}}$ has the same objects and
$\underline{\operatorname{Hom}}_{\mathcal C^{\mathrm{op}}}(x,y)
:=\underline{\operatorname{Hom}}_{\mathcal C}(y,x)$, with composition
defined using the symmetry of $\mathbf V$.
\end{definition}

\begin{definition}[enriched end]\label{def:enriched-end}

Let $\mathbf{V}$ be a complete closed symmetric monoidal category, with internal
Hom denoted by $\underline{\operatorname{Hom}}_{\mathbf{V}}$, let
$\mathcal{C}$ be a small $\mathbf{V}$-enriched category and let
$T:\mathcal{C}^{\mathrm{op}}\otimes\mathcal{C}\to\mathbf{V}$ be a
$\mathbf{V}$-enriched bifunctor. Completeness and closedness of
$\mathbf V$, together with smallness of $\mathcal C$, ensure that the
$\mathbf{V}$-valued enriched end of $T$ exists and is the
equalizer
\[
\int_{x\in\mathcal{C}}
T(x,x)
\longrightarrow
\prod_{x\in\mathcal{C}}
T(x,x)
\rightrightarrows
\prod_{x,y\in\mathcal{C}}
\underline{\operatorname{Hom}}_{\mathbf{V}}\bigl(
\underline{\operatorname{Hom}}_{\mathcal{C}}(x,y),
T(x,y)
\bigr),
\]
where the two parallel arrows are induced by the covariant action
$T(x,-)$ and the contravariant action $T(-,y)$ of $\mathcal{C}$;
equivalently, it is the object
\[
Z:=\int_{x\in\mathcal{C}}
T(x,x)
\]
equipped with a universal $\mathbf{V}$-natural wedge
\[
\lambda_x:Z\to T(x,x)
\]
(\cite[Section 2.1]{Kelly1982}).

For $\mathbf{V}$-enriched functors $F,G:\mathcal{C}\to\mathcal{D}$,
their enriched natural-transformation object is the end
\[
\int_{x\in\mathcal{C}}
\underline{\operatorname{Hom}}_{\mathcal{D}}\bigl(F(x),G(x)\bigr)
\]
obtained by taking
$T(x,y)=\underline{\operatorname{Hom}}_{\mathcal{D}}(F(x),G(y))$.
We use this construction with $\mathbf{V}=\mathcal{V}$ of
Definition~\ref{def:category-svect} with the closed symmetric monoidal
structure of Definition~\ref{def:simplicial-vector-space}, and with
$\mathbf V=\mathsf{sSet}$ under its cartesian monoidal structure
for simplicial functor categories.
\end{definition}

\section{Classifying spaces and $n$-representation categories}\label{sec:classifying-spaces}

This section introduces the classifying space of a simplicial group and
the $n$-representation categories studied in the rest of the paper.

\subsection{Classifying spaces}\label{sec:classifying-space-basic}

\begin{definition}\label{def:BG}
Let $G$ be a simplicial group. The simplicial category
$B^{\mathrm{simp}}G$ has a single object $\ast$ and
\[
B^{\mathrm{simp}}G(\ast,\ast):=G.
\]
Composition is degreewise multiplication in $G$, and the identity is
its unit in each degree. A homomorphism $G\to H$ induces a simplicial
functor $B^{\mathrm{simp}}G\to B^{\mathrm{simp}}H$. The classifying
space is
\[
\mathbf{B}G
:=
N^{\mathrm{hc}}\bigl(B^{\mathrm{simp}}G\bigr).
\]
\end{definition}

\begin{remark}\label{rem:bg-kan-wbar}

The simplicial set $\mathbf{B}G$ has exactly one vertex, since the
$0$-simplices of $N^{\mathrm{hc}}(B^{\mathrm{simp}}G)$ are the objects
of $B^{\mathrm{simp}}G$. It is a Kan complex. Indeed, the mapping
simplicial set $B^{\mathrm{simp}}G(\ast,\ast)=G$ is Kan by Moore's
theorem (\cite[Lemma I.3.4]{GoerssJardine}), so
$N^{\mathrm{hc}}(B^{\mathrm{simp}}G)$ is an $\infty$-category by
Proposition~\ref{prop:hc-nerve-properties}(1). By
Proposition~\ref{prop:hc-nerve-properties}(2), its homotopy category
has one object and endomorphism group
$\pi_0\operatorname{Map}_{B^{\mathrm{simp}}G}(\ast,\ast)=\pi_0(G)$;
hence it is a groupoid. By
\cite[Proposition 1.2.5.1]{HTT} the coherent nerve is a Kan complex.
Thus $\mathbf{B}G$ is a connected Kan complex.

The comparison with the classical bar construction is a natural
zigzag of weak equivalences.
Arakawa's \cite[Theorem~3.6]{Arakawa2023} compares its bar construction
with its version $N_{\mathrm{alt}}$ of the coherent nerve; the
left--right convention for $\overline W G$ is reconciled by passing
to $G^{\mathrm{op}}$ and using the natural inversion isomorphism
\cite[Remark~2.2]{Arakawa2023}. To pass to the present
$N^{\mathrm{hc}}=N_{\mathrm{std}}$, use the natural weak equivalence of
simplicial groups $G\to\operatorname{Sing}|G|$. It induces weak
equivalences on both coherent nerves, and the two nerve conventions
are naturally isomorphic for $B^{\mathrm{simp}}\operatorname{Sing}|G|$
\cite[Introduction, footnote~1]{Arakawa2023}. Every intermediate
nerve is Kan by the preceding argument. Hence the zigzag consists
of weak equivalences between Kan complexes and identifies
$\mathbf BG$ naturally with the homotopy type of $\overline W G$.

\end{remark}

\textbf{Convention.} For $m\ge1$, a map $f:G\to H$ of
simplicial groups is called an $m$-equivalence if $\pi_i(f)$ is an
isomorphism for $0\le i<m$ and $\pi_m(f)$ is surjective; in degree
zero this means an isomorphism of the component groups $\pi_0G$
and $\pi_0H$. The same term is used for a map between connected
Kan complexes when it
induces isomorphisms on $\pi_i$ for $0\le i<m$ and a surjection on
$\pi_m$, at the specified basepoints. This convention is used throughout
the paper.

\begin{proposition}\label{prop:bg-homotopy}
Let $f:G\to H$ be a map of simplicial groups.
\begin{enumerate}
\item $\pi_k(\mathbf{B}G)\cong\pi_{k-1}(G)$ for $k\ge1$.
\item $\pi_0(\mathbf{B}f)$ is an isomorphism.
\item $\pi_k(\mathbf{B}f)\cong\pi_{k-1}(f)$ for $k\ge1$.
\item for every $n\ge0$, by the convention stated above, $f$ is an
  $(n+1)$-equivalence if and only if $\mathbf{B}f$ is an
  $(n+2)$-equivalence.
\end{enumerate}
\end{proposition}

\begin{proof}

All homotopy groups are based at the identity vertex of $G$ and the
unique vertex of its classifying spaces. Use the right-action bar
construction of \cite[Construction~2.1 and Corollary~2.7]{Arakawa2023}.
Its principal fibration
$G\to WG\to\overline W_{\mathrm{Ar}}G$ has contractible total space.
The long exact homotopy sequence therefore gives natural group
isomorphisms
$\pi_k(\overline W_{\mathrm{Ar}}G)\cong\pi_{k-1}(G)$ for $k\ge2$.

The degree-one case requires checking the noncommutative group law.
In Arakawa's reduced bar construction a $1$-simplex is $a\in G_0$,
and a $2$-simplex $(a,z)\in G_0\times G_1$ has faces
$a$, $a\,d_0z$, and $d_1z$. Its fundamental-group relation is
$[a\,d_0z]=[a][d_1z]$. Taking $z=s_0b$ gives
$[ab]=[a][b]$, while taking $a=1$ gives
$[d_0z]=[d_1z]$. Hence the labeling map factors as a group
isomorphism $\pi_0G\cong\pi_1(\overline W_{\mathrm{Ar}}G)$:
the $1$-simplices supply generators and these $2$-simplices give
exactly the component and multiplication relations. The
construction is natural in $G$.

The natural Kan weak-equivalence zigzag of
Remark~\ref{rem:bg-kan-wbar} transports these isomorphisms to
$\mathbf BG$, proving (1), and identifies the map induced by $f$
with $\pi_{k-1}(f)$ for every $k\ge1$, proving (3).
Both $\mathbf{B}G$ and $\mathbf{B}H$ have exactly one vertex, so
$\pi_0(\mathbf{B}f):\ast\to\ast$ is automatically an isomorphism; this
is (2). Finally, by the convention stated before this
proposition, items (2) and (3) translate the defining conditions for
$f$ being an $(n+1)$-equivalence into the defining conditions for
$\mathbf{B}f$ being an $(n+2)$-equivalence. This proves (4).
\end{proof}

\subsection{$n$-representations}
\label{sec:truncated-rep-spaces}

Recall from Definition~\ref{def:n-linear-space} that
$\mathbf{Vect}^{(n)}_k$ is the coefficient target
for height $n$.

\begin{definition}[$n$-representation $\infty$-category]
\label{def:n-representation-category}
Let $n\ge1$ and let $G$ be any simplicial group.  The $\infty$-category of
$n$-representations of $G$ is
\[
\mathbf{Rep}^{(n)}_k(G)
:=
\operatorname{Fun}\bigl(\mathbf BG,\mathbf{Vect}^{(n)}_k\bigr).
\]
Here $\mathbf BG$ is a Kan complex by Remark~\ref{rem:bg-kan-wbar}, and
$\mathbf{Vect}^{(n)}_k$ is the full $\infty$-subcategory of
Definition~\ref{def:n-linear-space}, so the functor category is defined.
When $G$ is a homotopy $n$-group, we call its objects the
$n$-representations of the homotopy $n$-group $G$. No truncation
hypothesis on $G$ is needed for the definition.
\end{definition}

\noindent\textbf{Failure of closure under the levelwise tensor product.}
For $n\ge2$, $\mathbf{Vect}^{(n)}_k$ is not closed under the
levelwise tensor inherited from $\mathcal V$. Indeed, for
$X=N^{-1}(S^{n-1}k)$, Remark~\ref{rem:dold-kan-monoidal} gives
\[
N(X\otimes_{\mathcal V}X)\simeq
S^{n-1}k\otimes_k S^{n-1}k\cong S^{2n-2}k.
\]
Since $2n-2\ge n$, the tensor lies outside the prescribed homology range.
For $n=1$, the degree-zero homology range is closed under this tensor:
the field-valued Künneth formula and the normalized shuffle show that
the tensor of two complexes with homology only in degree zero again
has homology only in degree zero.
A separately defined truncated tensor would require its own coherence
construction.

\begin{definition}\label{def:full-rep-cat}
The untruncated representation $\infty$-category is
\[
\mathcal{R}_k(G)
:=
\operatorname{Fun}\bigl(\mathbf{B}G,\mathcal{V}^{\infty}\bigr).
\]
This is defined for every simplicial group $G$ and imposes no
finite homological-height bound on the coefficient objects.
\end{definition}

\begin{remark}\label{rem:bg-naturality}

A simplicial-group map $f:G\to H$ induces $B^{\mathrm{simp}}f$ and hence
$\mathbf Bf:\mathbf BG\to\mathbf BH$. Precomposition defines pullback
functors
\[
f^*:\mathcal R_k(H)\longrightarrow\mathcal R_k(G).
\]
For $n\ge1$, the same construction gives
\[
f^*:\mathbf{Rep}^{(n)}_k(H)
\longrightarrow\mathbf{Rep}^{(n)}_k(G).
\]
This construction is strictly contravariantly functorial:
$\operatorname{id}_G^*=\operatorname{id}$ and
$(g\circ f)^*=f^*\circ g^*$ for composable simplicial-group maps,
because $B^{\mathrm{simp}}$ and $N^{\mathrm{hc}}$ preserve composition
and these restriction functors are literal precomposition.
Proposition~\ref{prop:bg-homotopy} records the corresponding
naturality on homotopy groups.

\end{remark}

\begin{definition}\label{def:fiber-functor}
Evaluation at the unique vertex $\ast$ of $\mathbf BG$ sends an
untruncated representation $F\in\mathcal R_k(G)$ to its underlying
object $F(\ast)\in\mathcal V^{\infty}$. For $n\ge1$, its restriction
defines the $n$-representation fiber functor
\[
\omega_G^{(n)}:
\mathbf{Rep}^{(n)}_k(G)\longrightarrow\mathbf{Vect}^{(n)}_k,
\qquad
F\longmapsto F(\ast).
\]
Because $\mathbf Bf$ sends the unique vertex to the unique vertex, the
pullback functors of Remark~\ref{rem:bg-naturality} satisfy the canonical
equality
\[
\omega_G^{(n)}\circ f^*=\omega_H^{(n)}.
\]
This is an equality of functors in the chosen model, including on
higher simplices: evaluating $F\circ\mathbf Bf$ at $\ast$ is
evaluating $F$ at $\mathbf Bf(\ast)=\ast$. The same evaluation
compatibility holds for untruncated representations.
\end{definition}

For an ordinary group $\Gamma$, write
$\operatorname{Rep}^{\mathrm{ord}}_k(\Gamma)$ for the ordinary category
of all $k$-vector spaces with a $\Gamma$-action and equivariant linear
maps, equivalently all left modules over the ordinary group algebra
$k[\Gamma]$. The correspondence extends the action linearly from the
group basis of $k[\Gamma]$ and leaves equivariant maps unchanged.
No finite-dimensionality assumption is imposed.

\begin{proposition}[ordinary representations in height one]
\label{prop:one-representations}
For every simplicial group $G$, there is an equivalence compatible with the
fiber functors
\[
\mathbf{Rep}^{(1)}_k(G)
\simeq
N\bigl(\operatorname{Rep}^{\mathrm{ord}}_k(\pi_0G)\bigr).
\]
In particular, if $G$ is a homotopy $1$-group, its $1$-representations are
ordinary linear representations of its constant strict model $\pi_0G$.
\end{proposition}

\begin{proof}
Choose the equivalence
$\mathbf{Vect}^{(1)}_k\simeq N(\mathsf{Vect}_k)$ of
Proposition~\ref{prop:n-linear-category-height}. For an
$\infty$-category $A$ and an ordinary category $D$, functors into
$N(D)$ and their natural transformations are determined by $hA$,
giving $\operatorname{Fun}(A,N(D))\simeq
N(\operatorname{Fun}(hA,D))$. Here $h(\mathbf BG)$ is the one-object
groupoid with automorphism group
$\pi_1(\mathbf BG)\cong\pi_0G$ as groups by
Proposition~\ref{prop:bg-homotopy}. Its functors to $\mathsf{Vect}_k$
and natural transformations are precisely ordinary representations
and intertwiners. Under the chosen coefficient equivalence, evaluation
at the unique vertex corresponds to the ordinary forgetful functor.
\end{proof}

\begin{table}[htbp]
\centering
\small
\renewcommand{\arraystretch}{1.25}
\begin{tabularx}{\textwidth}{|c|X|X|X|}
\hline
height & group model & linear target & representation model \\
\hline
$n=1$ & $G\simeq\operatorname{const}\pi_0G$
& a vector space in degree $0$
& an ordinary action $\pi_0G\to\operatorname{GL}(V)$ \\
\hline
$n=2$ & a Moore-length-one simplicial group, corresponding to a crossed
module after choosing the strict model
& a two-term complex $V_1\to V_0$, equivalently an internal category in
vector spaces
& a homotopy coherent action; a crossed-module morphism only after compatible
strict models have been chosen \\
\hline
general $n$ & an $(n-1)$-truncated simplicial group
& a complex representative concentrated in degrees $0,\ldots,n-1$
& a functor $\mathbf BG\to\mathbf{Vect}^{(n)}_k$ \\
\hline
\end{tabularx}
\caption{Low-dimensional models for $n$-representations. The group-model
column assumes that $G$ is a homotopy $n$-group. For arbitrary $G$, it describes
the $(n-1)$-truncated source model $Q_nG$ of
Proposition~\ref{lem:visible-source-reduction}, rather than $G$ itself.}
\label{tab:n-representation-low-models}
\end{table}

\begin{remark}[comparison with crossed-module models]\label{rem:n-rep-strict-model-boundary}
An object of $\mathbf{Rep}^{(2)}_k(G)$ is intrinsically a homotopy coherent
functor $\mathbf BG\to\mathbf{Vect}^{(2)}_k$.  Suppose, in addition, that a
strict point-set action $G'\to\operatorname{Aut}(V')$ representing this object
has been chosen, and that its source and target are compatible simplicial
groups of Moore length at most one.  Theorem~2.2 of
\cite{MutluPorter1998} then sends this strict morphism to a morphism of
crossed modules. The strict diagram models of
Section~\ref{sec:rep-module-categories} realize coherent representations
using projectively bifibrant diagrams; their coefficient condition is specified in
Proposition~\ref{prop:n-representation-function-model}.
\end{remark}

\subsection{The group algebra functor}\label{sec:group-algebra-functor}

\begin{definition}\label{def:group-algebra}
The simplicial group algebra $k[G]$ is the simplicial vector space
whose $n$-simplices form the free $k$-vector space $k[G_n]$ on the set
$G_n$. The face and degeneracy maps are the linear extensions
\[
k[d_i]:k[G_{n+1}]\to k[G_n],
\qquad
k[s_i]:k[G_n]\to k[G_{n+1}]
\]
of the corresponding maps $d_i:G_{n+1}\to G_n$ and
$s_i:G_n\to G_{n+1}$ of $G$. The multiplication is the bilinear
extension of the group multiplication $G_n\times G_n\to G_n$, and the
unit is the element of $k[G_n]$ corresponding to the identity of
$G_n$. Since faces and degeneracies are group homomorphisms, their
linear extensions preserve multiplication and units; thus these maps
make $k[G]$ a simplicial associative $k$-algebra.
A simplicial-group map $f:G\to H$ extends degreewise and linearly to a
simplicial algebra map $k[f]:k[G]\to k[H]$. This construction is
functorial: $k[\operatorname{id}_G]=\operatorname{id}_{k[G]}$ and
$k[g\circ f]=k[g]\circ k[f]$, as follows by checking on group-basis
elements.
\end{definition}

\subsection{Source truncation and homotopy invariance}
\label{sec:truncated-invariance}

The argument in this subsection is independent of the field and of
Dold--Kan.  It first isolates the source-truncation statement for an
arbitrary locally $(n-1)$-truncated $\infty$-category and then applies it to the natural target
$\mathbf{Vect}^{(n)}_k$.

\begin{proposition}[source-truncation identity]\label{prop:A3}
Let $n\ge0$, let $u:X\to Y$ be a map of Kan complexes, and let
$\mathcal D$ be a locally $(n-1)$-truncated $\infty$-category.  Write
$a_X:X\to\pi_{\le n}X$ and $a_Y:Y\to\pi_{\le n}Y$ for the natural
truncation maps. As functors from $\operatorname{Fun}(\pi_{\le n}Y,\mathcal D)$
to $\operatorname{Fun}(X,\mathcal D)$, there is an identity
\[
u^*\circ a_Y^*
=a_X^*\circ\bigl(\pi_{\le n}u\bigr)^*.
\]
Consequently, $u^*:\operatorname{Fun}(Y,\mathcal D)\to
\operatorname{Fun}(X,\mathcal D)$ is an equivalence if and only if
precomposition with $\pi_{\le n}u$ is an equivalence.
\end{proposition}

\begin{proof}
Naturality of homotopy truncation gives
$\pi_{\le n}u\circ a_X=a_Y\circ u$.  Precomposition therefore gives a
commutative square
\[
\begin{CD}
\operatorname{Fun}(Y,\mathcal D) @>{u^*}>>
\operatorname{Fun}(X,\mathcal D)\\
@A{a_Y^*}A{\simeq}A @A{a_X^*}A{\simeq}A\\
\operatorname{Fun}(\pi_{\le n}Y,\mathcal D)
@>{(\pi_{\le n}u)^*}>>
\operatorname{Fun}(\pi_{\le n}X,\mathcal D).
\end{CD}
\]
The vertical functors are equivalences, including when $n=0$, by
Proposition~\ref{prop:homotopy-n-universal}.  The displayed identity and
two-out-of-three prove the final assertion.
If $\mathcal D$ is a strict $(n,1)$-model, the vertical
functors are isomorphisms and the identity may also be written
$u^*=a_X^*\circ(\pi_{\le n}u)^*\circ(a_Y^*)^{-1}$.
\end{proof}

\begin{theorem}[source-truncation invariance]
\label{thm:source-truncation-invariance}
Let $n\ge0$, let $u:X\to Y$ be a map of Kan complexes, and let
$\mathcal D$ be a locally $(n-1)$-truncated $\infty$-category.  If
\[
\pi_{\le n}u:\pi_{\le n}X\xrightarrow{\ \simeq\ }\pi_{\le n}Y
\]
is an equivalence, then precomposition induces an equivalence
\[
u^*:\operatorname{Fun}(Y,\mathcal D)
\xrightarrow{\ \simeq\ }
\operatorname{Fun}(X,\mathcal D).
\]
\end{theorem}

\begin{proof}
Since $\pi_{\le n}u$ is an equivalence of $\infty$-groupoids,
Lemma~\ref{lem:precomp-equivalence} makes
$(\pi_{\le n}u)^*$ an equivalence.  The two vertical functors in the
natural square of Proposition~\ref{prop:A3} are also equivalences;
two-out-of-three gives the assertion.
\end{proof}

\begin{theorem}[$n$-representation homotopy invariance]\label{thm:A1}
Let $n\ge1$ and let $f:G\to H$ be a map of simplicial groups.  Suppose
that
\[
\pi_i(f):\pi_i(G)\longrightarrow\pi_i(H)
\]
is an isomorphism for $0\le i\le n-1$.  Then pullback induces an
equivalence
\[
f^*:
\mathbf{Rep}^{(n)}_k(H)
\xrightarrow{\ \simeq\ }
\mathbf{Rep}^{(n)}_k(G).
\]
\end{theorem}

\begin{proof}
The classifying spaces are connected Kan complexes by
Remark~\ref{rem:bg-kan-wbar}.  Proposition~\ref{prop:bg-homotopy} gives
\[
\pi_j(\mathbf Bf)\cong\pi_{j-1}(f)
\qquad(1\le j\le n),
\]
For $j=1$, this is an isomorphism of groups, since the hypothesis at
$i=0$ is a group isomorphism and the classifying-space comparison
preserves multiplication. Both spaces have a single component, so
Definition~\ref{def:homotopy-n-truncation}(5) identifies
$\pi_{\le n}(\mathbf Bf)$ as an equivalence.  By
Proposition~\ref{prop:n-linear-category-height},
$\mathbf{Vect}^{(n)}_k$ is locally $(n-1)$-truncated.  Apply
Theorem~\ref{thm:source-truncation-invariance} with
$X=\mathbf BG$, $Y=\mathbf BH$, and
$\mathcal D=\mathbf{Vect}^{(n)}_k$.
\end{proof}

\begin{corollary}\label{cor:A2}
Let $n\ge1$ and let $f:G\to H$ be an $n$-equivalence in the convention
of Section~\ref{sec:classifying-space-basic}; that is,
\[
\pi_i(f)\text{ is an isomorphism for }0\le i\le n-1,
\qquad
\pi_n(f)\text{ is surjective}.
\]
Then $f^*:\mathbf{Rep}^{(n)}_k(H)\to
\mathbf{Rep}^{(n)}_k(G)$ is an equivalence.
\end{corollary}

\begin{proof}
An $n$-equivalence has the full isomorphism range required by
Theorem~\ref{thm:A1}, which gives the result. Its additional
$\pi_n$-surjectivity condition plays no role in the argument.
\end{proof}

\begin{proposition}[reduction to the truncated source]
\label{lem:visible-source-reduction}
For $n\ge1$, put
\[
Q_nG:=\operatorname{cosk}^{\mathsf{sGrp}}_{n}(G|_{\le n}),
\]
a functorial simplicial-group model for $\tau_{\le n-1}G$
\cite[Example~3.5.7.24, Tag~\href{https://kerodon.net/tag/055B}{055B}]{Kerodon}.  For every
locally $(n-1)$-truncated $\infty$-category $\mathcal D$,
precomposition with the canonical map
$\mathbf BG\to\mathbf BQ_nG$ induces an equivalence
\[
\operatorname{Fun}(\mathbf BQ_nG,\mathcal D)
\xrightarrow{\ \simeq\ }
\operatorname{Fun}(\mathbf BG,\mathcal D).
\]
In particular,
\[
\mathbf{Rep}^{(n)}_k(Q_nG)
\xrightarrow{\ \simeq\ }
\mathbf{Rep}^{(n)}_k(G)
\]
compatibly with the fiber functors.  If $G$ is a homotopy $n$-group,
then $G\to Q_nG$ is a weak equivalence.
\end{proposition}

\begin{proof}
Let $U:\mathsf{sGrp}\to\mathsf{sSet}$ be the degreewise forgetful functor
and $F$ its degreewise free-group left adjoint. These adjunctions also
hold for $n$-truncated objects, and restriction commutes with $F$.
Successive adjunctions therefore give, naturally in a simplicial set $X$,
\[
\operatorname{Hom}_{\mathsf{sSet}}(X,UQ_nG)
\cong\operatorname{Hom}_{\mathsf{sGrp}_{\le n}}
  ((FX)|_{\le n},G|_{\le n})
\cong\operatorname{Hom}_{\mathsf{sSet}}
  (X,\operatorname{cosk}^{\mathsf{sSet}}_n UG).
\]
Thus $UQ_nG\cong\operatorname{cosk}^{\mathsf{sSet}}_n UG$, carrying the
canonical map to the canonical coskeleton map. Since $UG$ is Kan,
\cite[Example~3.5.7.24, Tag~\href{https://kerodon.net/tag/055B}{055B}]{Kerodon}, applied with its parameter $n-1$,
shows that $G\to Q_nG$ is an isomorphism on
$\pi_0,\ldots,\pi_{n-1}$ and the target has no higher homotopy groups.
By Proposition~\ref{prop:bg-homotopy}, the induced map
$\mathbf BG\to\mathbf BQ_nG$ is an isomorphism on homotopy groups
through degree $n$, and its target is an $n$-type.  It therefore exhibits
$\mathbf BQ_nG$ as $\pi_{\le n}(\mathbf BG)$.  The first assertion is
Theorem~\ref{thm:source-truncation-invariance}; the representation
statement uses Proposition~\ref{prop:n-linear-category-height}, and
evaluation compatibility is pointwise.  If $G$ is a homotopy $n$-group,
the truncation map is an isomorphism on every possibly nonzero homotopy
group, hence a weak equivalence of Kan complexes.
\end{proof}

\begin{remark}\label{rem:A3}
Proposition~\ref{prop:A3} is an identity decomposition, not a converse
to Theorem~\ref{thm:source-truncation-invariance}.  An equivalence of
functor categories need not force $\pi_{\le n}u$ to be an equivalence:
take the inclusion $u:\Delta^0\hookrightarrow\Delta^0\sqcup\Delta^0$
of one component and the terminal target $\mathcal D=\Delta^0$.
Precomposition is an isomorphism of terminal functor categories, whereas
$\pi_{\le n}u=u$ fails already on components for every $n\ge0$.
Section~\ref{sec:detection} proves degree-zero conservativity for
$n$-representations and gives precise coefficient-dependent failures in
positive degree.
\end{remark}

\section{Representation $\infty$-categories and module categories}\label{sec:rep-module-categories}

This section compares strict representation categories, representation
$\infty$-categories, and module categories. These comparisons give concrete models for coherent actions and their
mapping spaces, including the finite-height subcategories studied here.

\subsection{Enriched representations as strict functor categories}\label{sec:enriched-strict-functors}

A strict action of $G$ on a simplicial vector space defines a
representation. Equivariant maps between these representations form
the hom-objects of a $\mathcal V$-enriched category.

\paragraph{\textbf{Strict actions.}}
For $V\in\mathcal V$, composition makes $\operatorname{Map}(V,V)$
a simplicial monoid. Let $\operatorname{Aut}(V)_m$ be its group of
units in degree $m$. These are invertible parameterized maps
$V\otimes k[\Delta^m]\to V$, with composition using the diagonal
of $\Delta^m$.
Simplicial operators preserve composition and units, so the groups
$\operatorname{Aut}(V)_m$ form a simplicial group.
A strict representation is a pair $(V,\rho)$ with
\[
\rho:G\longrightarrow\operatorname{Aut}(V)
\]
a simplicial group homomorphism. These pairs are the objects of
$\mathbf{Rep}_k^{\mathrm{enr}}(G)$.

\paragraph{\textbf{Equivariant maps.}}
For two representations $(V,\rho)$ and $(W,\sigma)$, the
tensor--mapping-space adjunction gives the linearly extended actions
\[
\alpha:k[G]\otimes_{\mathcal V}V\longrightarrow V,
\qquad
\beta:k[G]\otimes_{\mathcal V}W\longrightarrow W.
\]
A simplicial linear map $h:V\to W$ is equivariant when
\[
h\circ\alpha=\beta\circ(1\otimes h).
\]
The equivariant internal Hom imposes this condition on parameterized
maps through the equalizer in $\mathcal V$
\[
\underline{\operatorname{Hom}}_{\mathbf{Rep}}((V,\rho),(W,\sigma))
\longrightarrow
\underline{\operatorname{Hom}}_{\mathcal V}(V,W)
\rightrightarrows
\underline{\operatorname{Hom}}_{\mathcal V}
  (k[G]\otimes_{\mathcal V}V,W).
\]
The two arrows encode, respectively, applying the source action
before the map and applying the target action after it.
Explicitly, they are adjoint to
\[
\underline{\operatorname{Hom}}_{\mathcal V}(V,W)\otimes k[G]\otimes V
\xrightarrow{\,1\otimes\alpha\,}
\underline{\operatorname{Hom}}_{\mathcal V}(V,W)\otimes V
\xrightarrow{\,\operatorname{ev}\,}W
\]
and
\[
\begin{aligned}
\underline{\operatorname{Hom}}_{\mathcal V}(V,W)\otimes k[G]\otimes V
&\xrightarrow{\,\mathrm{swap}\otimes1\,}
k[G]\otimes\underline{\operatorname{Hom}}_{\mathcal V}(V,W)\otimes V\\
&\xrightarrow{\,1\otimes\operatorname{ev}\,}
k[G]\otimes W
\xrightarrow{\,\beta\,}W.
\end{aligned}
\]
Here tensor products are identified using the associativity constraints.
Completeness of $\mathcal V$ supplies the equalizer. Equivariant maps
are closed under composition and include the identities, so the
composition and units of $\mathcal V$ give
$\mathbf{Rep}_k^{\mathrm{enr}}(G)$ its $\mathcal V$-enrichment.

\paragraph{\textbf{The functor description.}}
A simplicial functor from the single-object category
$B^{\mathrm{simp}}G$ chooses $V$ and an action preserving multiplication
and the unit. Its natural transformations satisfy the equivariance
condition above. Write
$\operatorname{Fun}_{\Delta}(B^{\mathrm{simp}}G,\mathcal V)$ for this
simplicial functor category, equipped with the $\mathcal V$-valued
equivariant hom-objects above.

\begin{theorem}\label{thm:rep-strict-functor}
There is an isomorphism of $\mathcal V$-enriched categories
\[
\mathbf{Rep}_k^{\mathrm{enr}}(G)
\;\cong\;
\operatorname{Fun}_{\Delta}\bigl(B^{\mathrm{simp}}G,\mathcal{V}\bigr).
\]
\end{theorem}

\begin{proof}
A simplicial functor
$F:B^{\mathrm{simp}}G\to\mathcal V$ is determined by the
object $V:=F(\ast)$ and by a simplicial monoid homomorphism
\[
G
\longrightarrow
\operatorname{Map}_{\mathcal{V}}(V,V)
=
\operatorname{Und}\bigl(\underline{\operatorname{Hom}}_{\mathcal{V}}(V,V)\bigr)
\]
that is, a map preserving composition and the unit. Since every
$g\in G_n$ has an inverse $g^{-1}$ and $F$ preserves composition,
\[
F(g)\circ F(g^{-1})
=
F(gg^{-1})
=
F(e)
=
\mathrm{id}_V,
\]
and similarly $F(g^{-1})\circ F(g)=\mathrm{id}_V$ in
$\operatorname{Map}_{\mathcal V}(V,V)_n$. Hence the simplicial monoid map
factors uniquely through $\operatorname{Aut}(V)$. Thus
$F$ determines a strict representation $(V,\rho)$ with
$\rho(g):=F(g)$.

Conversely, given a strict representation $(V,\rho)$, define
$F(\ast):=V$ and $F(g):=\rho(g)$ for $g\in G$. Because $\rho$ is a
simplicial group homomorphism, this assignment preserves composition
and identities on $B^{\mathrm{simp}}G(\ast,\ast)=G$, and hence defines
a simplicial functor. The two
constructions are mutually inverse.

For corresponding objects, both enriched categories now have, by
definition, the same equalizer as their hom-object. The object
identification therefore induces the identity on every hom-object.
These identities respect composition and units because both enrichments
restrict those of $\mathcal V$. Hence the mutually inverse object
constructions form an isomorphism of $\mathcal V$-enriched categories.

The compatibility with the usual simplicial enrichment follows from
the fact that
$\operatorname{Und}:\mathcal V\to\mathsf{sSet}$ creates limits:
limits are computed levelwise, and
$\mathsf{Vect}_k\to\mathsf{Set}$ creates limits. Thus the underlying
simplicial set of the equivariant-Hom equalizer is the
enriched-natural-transformation end
\[
\int_{x\in B^{\mathrm{simp}}G}
\operatorname{Map}(F(x),F'(x)).
\]
The underlying simplicial category is therefore the standard
simplicial functor category.
\end{proof}

\subsection{The function model theorem}\label{sec:function-model}

For a simplicial category $\mathbf{A}$ and a small simplicial category
$\mathcal{C}$, write
$\mathbf{A}^{\mathcal{C}}:=\operatorname{Fun}_\Delta(\mathcal{C},\mathbf{A})$
for the simplicial category of simplicial functors
$\mathcal{C}\to\mathbf{A}$. Its objects are simplicial functors, and
its mapping spaces are the ends of
Definition~\ref{def:enriched-end} applied with
$\mathbf V=\mathsf{sSet}$:
\[
\operatorname{Map}_{\mathbf{A}^{\mathcal{C}}}(F,G)
=
\int_{x\in\mathcal{C}}\operatorname{Map}_{\mathbf{A}}(F(x),G(x)).
\]
Here the end is taken in $\mathsf{sSet}$; it exists because
$\mathcal C$ is small and $\mathsf{sSet}$ is complete.

In the special case $\mathbf A=\mathcal V$ and
$\mathcal C=B^{\mathrm{simp}}G$, Section~\ref{sec:enriched-strict-functors}
identifies this mapping space with the underlying simplicial set of a
$\mathcal V$-valued equivariant-Hom equalizer.
In particular, the underlying ordinary category has the simplicial
natural transformations as morphisms, i.e. the $0$-simplices of these
mapping spaces.

If $\mathbf A$ is a combinatorial simplicial model category, the
diagram category $\mathbf A^{\mathcal C}$ admits a combinatorial
simplicial projective model structure.

We apply
\cite[Proposition A.3.3.2]{HTT} with enrichment base
$\mathcal S=\mathsf{sSet}$, equipped with the Kan model structure
and Cartesian product. This base is an excellent model category by
\cite[Example A.3.2.18]{HTT}. In the terminology of
\cite[Definition A.3.1.5]{HTT}, an $\mathsf{sSet}$-enriched model
category is a simplicial model category. Thus the remaining inputs
are the combinatorial simplicial model category $\mathbf A$ and
the small simplicial category $\mathcal C$.

\paragraph{\textbf{The projective structure.}}
For a natural transformation $u:F\to F'$ in $\mathbf A^{\mathcal C}$,
the weak equivalences and projective fibrations are determined
objectwise:
\[
\begin{aligned}
u\text{ is a weak equivalence}
&\ \Longleftrightarrow\
u_x:F(x)\to F'(x)\text{ is a weak equivalence for every }x,\\
u\text{ is a projective fibration}
&\ \Longleftrightarrow\
u_x:F(x)\to F'(x)\text{ is a fibration for every }x.
\end{aligned}
\]
A projective cofibration has the left lifting property against
all maps that are both projective fibrations and weak equivalences
\cite[Definition A.3.3.1]{HTT}. These lifts are natural transformations,
so the lifting condition includes compatibility with the diagram
structure.
The simplicial tensor is computed objectwise,
$(F\otimes K)(x)=F(x)\otimes K$, and
\cite[Remark A.3.3.4]{HTT} supplies its compatibility with the
projective model structure.

\paragraph{\textbf{Application to strict representations.}}
Here $\mathbf A=\mathcal V$ is combinatorial and simplicial by
Proposition~\ref{prop:model-structure-svect}, with mapping spaces
$\operatorname{Und}\underline{\operatorname{Hom}}_{\mathcal V}(V,W)$
and tensor $V\otimes K=V\otimes_{\mathcal V}k[K]$.
The indexing category $\mathcal C=B^{\mathrm{simp}}G$ has one object
and the small simplicial mapping space $G$.
Hence the theorem applies to the strict representation category
identified in Section~\ref{sec:enriched-strict-functors}.
For a map of strict representations, weak equivalence means that
the underlying normalized map is a quasi-isomorphism, and fibration
means that it is surjective in every positive degree.
Projective cofibrancy concerns lifts that respect the $G$-action.
It provides the appropriate source models for derived mapping
spaces and for the comparison with coherent representations below.

\begin{lemma}\label{lem:bifibrant-inclusion}
Give $\mathbf A^{\mathcal C}$ its projective simplicial
model structure, and write $(\mathbf A^{\mathcal C})^\circ$ for its full
simplicial subcategory of fibrant-cofibrant objects. If
$\mathbf A=\mathcal V$ and $\mathcal C=B^{\mathrm{simp}}G$, then every
projective diagram is fibrant, and
$(\mathbf A^{\mathcal C})^\circ$ is locally Kan.
\end{lemma}

\begin{proof}

Projective fibrations are objectwise fibrations. Every
object of $\mathcal V$ is fibrant by
Proposition~\ref{prop:model-structure-svect}, so every object of the projective
diagram model category is fibrant. Thus the objects of
$(\mathbf A^{\mathcal C})^\circ$ are precisely the projectively
cofibrant diagrams. For such diagrams
$F,F'$, the mapping simplicial set in the full subcategory is the same as
in $\mathbf A^{\mathcal C}$. By the $\mathcal V$-enrichment constructed
in Section~\ref{sec:enriched-strict-functors}, it is the underlying
simplicial set of the $\mathcal V$-valued equivariant-Hom equalizer. That
equalizer is a simplicial $k$-vector space, so its underlying simplicial
set is Kan by Proposition~\ref{prop:simplicial-vector-kan}. Hence the full
subcategory is locally Kan.

\end{proof}

\noindent\textbf{Warning.} The inclusion
$(\mathbf A^{\mathcal C})^\circ\hookrightarrow\mathbf A^{\mathcal C}$
need not be a Dwyer--Kan equivalence for the given simplicial enrichment. A model
weak-equivalence zigzag becomes invertible only after localization; it
need not be a simplicial homotopy equivalence.
For example, take $k=\mathbb F_p$ and constant $G=C_p$. The constant
trivial $k[G]$-module $M=k$ is not projective: the local ring
$k[C_p]\cong k[t]/(t-1)^p$ has only free finitely generated projectives,
whose $k$-dimensions are multiples of $p$.
A projectively cofibrant simplicial $k[G]$-module $Q$ has projective $Q_0$.
If $M$ were simplicially homotopy equivalent to such a $Q$, the
composite $M\to Q\to M$ would equal the identity in degree zero,
because $M$ is constant. This would make $M$ a retract of $Q_0$,
a contradiction. Thus the inclusion is not essentially surjective
on the homotopy category of that enrichment before localization, although
it is fully faithful on its mapping spaces.

\begin{theorem}\label{thm:function-model}
Set $\mathbf{A}:=\mathcal{V}$ and
$\mathcal{C}:=B^{\mathrm{simp}}G$, and give
$\mathbf A^{\mathcal C}$ the projective model structure. There is an
equivalence of
$\infty$-categories
\[
\Theta:
N^{\mathrm{hc}}\bigl((\mathbf{A}^{\mathcal{C}})^\circ\bigr)
\;\xrightarrow{\ \simeq\ }\;
\mathcal{R}_k(G)
:=
\operatorname{Fun}\bigl(\mathbf{B}G,\mathcal{V}^{\infty}\bigr).
\]
\end{theorem}

\begin{proof}
Apply \cite[Proposition 4.2.4.4]{HTT} with
$\mathbf{A}:=\mathcal{V}$, $\mathcal{C}:=B^{\mathrm{simp}}G$,
$K:=N^{\mathrm{hc}}(\mathcal C)=\mathbf BG$, and $U:=\mathbf A$.
The enrichment category is $\mathcal S:=\mathsf{sSet}$ with its
Kan model structure.

\emph{Source comparison.}
Both $\mathcal C$ and $K$ are small in the fixed universe.
The source $B^{\mathrm{simp}}G$ is locally Kan because its single mapping
space is $G$, and $G$ is Kan by
Proposition~\ref{prop:simplicial-group-kan}. By
Proposition~\ref{prop:hc-nerve-properties}(4), the counit
\[
\operatorname{Path}\bigl[N^{\mathrm{hc}}(B^{\mathrm{simp}}G)\bigr]
\longrightarrow B^{\mathrm{simp}}G
\]
is a Dwyer--Kan equivalence. Write
$u:\operatorname{Path}[K]\to\mathcal C$ for this map; it is the source
equivalence required by the cited proposition.

\emph{Model-category conditions.}
The category $\mathcal S$ is excellent
\cite[Example A.3.2.18]{HTT}. Proposition~\ref{prop:model-structure-svect}
supplies the combinatorial simplicial model structure on $\mathbf A$,
with functorial factorizations and cofibrant replacements. By
\cite[Definition A.3.4.1]{HTT}, a full subcategory is a chunk when the
functorial factorization properties (a) and (b) hold; by
\cite[Example A.3.4.4]{HTT}, $\mathbf{A}$ is a chunk of itself. Since
$\mathcal{C}:=B^{\mathrm{simp}}G$ is a small simplicial category,
\cite[Proposition A.3.3.2]{HTT} endows $\mathbf{A}^{\mathcal{C}}$ with
a projective combinatorial model structure, and
\cite[Remark A.3.3.4]{HTT} shows that this is again a simplicial model
structure; applying
\cite[Example A.3.4.4]{HTT} to $\mathbf{A}^{\mathcal{C}}$ with its
projective model structure shows that $\mathbf{A}^{\mathcal{C}}$ is a
chunk of itself. Thus, by \cite[Definition A.3.4.9]{HTT},
$U=\mathbf{A}$ is a $\mathcal{C}$-chunk of $\mathbf{A}$.

\emph{Target and conclusion.}

By Proposition~\ref{prop:strong-cofibrant}, every object of
$\mathcal{V}$ is both fibrant and cofibrant, so
$\mathbf{A}^{\circ}=\mathcal{V}$; hence
$N^{\mathrm{hc}}(\mathbf{A}^{\circ})=\mathcal{V}^{\infty}$.

The cited rectification proposition therefore gives
\[
N^{\mathrm{hc}}\bigl((\mathbf{A}^{\mathcal{C}})^{\circ}\bigr)
\;\simeq\;
\operatorname{Fun}\bigl(\mathbf{B}G,\mathcal{V}^{\infty}\bigr).
\]

\end{proof}

\begin{proposition}[$n$-representation function model]
\label{prop:n-representation-function-model}
Let $n\ge1$.  Write
\[
\bigl(\mathcal V^{B^{\mathrm{simp}}G}\bigr)^{\circ,(n)}
\]
for the full simplicial subcategory of projectively bifibrant diagrams
$F$ satisfying
\[
H_i\bigl(NF(\ast)\bigr)=0\qquad(i\ge n).
\]
The equivalence $\Theta$ of Theorem~\ref{thm:function-model} restricts to
an equivalence
\[
N^{\mathrm{hc}}
\Bigl(\bigl(\mathcal V^{B^{\mathrm{simp}}G}\bigr)^{\circ,(n)}\Bigr)
\xrightarrow{\ \simeq\ }
\mathbf{Rep}^{(n)}_k(G),
\]
and this equivalence is compatible, up to its canonical natural
equivalence, with evaluation at the unique object.
\end{proposition}

\begin{proof}
The comparison in HTT Proposition~4.2.4.4 is natural under restriction
of diagrams. At the unique vertex, its comparison is the ordinary
evaluation of a diagram at $\ast$, so it gives a naturally commutative
square
\[
\begin{CD}
N^{\mathrm{hc}}
\bigl((\mathcal V^{B^{\mathrm{simp}}G})^\circ\bigr)
@>{\Theta}>>
\operatorname{Fun}(\mathbf BG,\mathcal V^\infty)\\
@V{\operatorname{ev}_\ast}VV @VV{\operatorname{ev}_\ast}V\\
\mathcal V^\infty @= \mathcal V^\infty.
\end{CD}
\]
The subcategory $\mathbf{Vect}^{(n)}_k\subset\mathcal V^\infty$ is full
and replete, because membership is the weak-equivalence-invariant
condition $H_i(NV)=0$ for $i\ge n$.  Consequently the source displayed
in the statement is exactly the inverse image of
$\mathbf{Vect}^{(n)}_k$ under the left evaluation functor, while
$\mathbf{Rep}^{(n)}_k(G)$ is the inverse image under the right evaluation
functor.

Full faithfulness of the restriction follows from fullness of these
subcategories.  For essential surjectivity, choose a preimage under
$\Theta$ of an $n$-representation.  Evaluation compatibility identifies
its base-point value, up to equivalence, with an object of
$\mathbf{Vect}^{(n)}_k$; repleteness therefore puts the chosen preimage in
the indicated source subcategory.  This proves the restricted equivalence
and its evaluation compatibility.
\end{proof}

\subsection{Equivariant Hom computes derived mapping spaces}\label{sec:mapping-spaces-preserved}

\begin{theorem}\label{thm:mapping-spaces-preserved}
Let $(V,\rho),(W,\sigma)$ be strict representations whose
corresponding diagrams $F_{V,\rho},F_{W,\sigma}$ are projectively
cofibrant (and hence bifibrant). Then
\[
\operatorname{Und}\Bigl(\underline{\operatorname{Hom}}_{\mathbf{Rep}}
\bigl((V,\rho),(W,\sigma)\bigr)\Bigr)
\;\simeq\;
\operatorname{Map}_{\mathcal{R}_k(G)}
\bigl(\Theta F_{V,\rho},\Theta F_{W,\sigma}\bigr).
\]
\end{theorem}

\begin{proof}

After applying $\operatorname{Und}$,
Theorem~\ref{thm:rep-strict-functor} identifies the left-hand Kan complex
with the mapping simplicial set
\[
\operatorname{Map}_{(\mathcal V^{B^{\mathrm{simp}}G})^\circ}
(F_{V,\rho},F_{W,\sigma}).
\]
The two diagrams lie in the locally Kan simplicial category
$(\mathcal V^{B^{\mathrm{simp}}G})^\circ$. Hence the canonical comparison
map of \cite[Remark~4.6.7.6, Tag~\href{https://kerodon.net/tag/01LF}{01LF}]{Kerodon} is a homotopy
equivalence from this
Kan complex to
\[
\operatorname{Map}_{N^{\mathrm{hc}}
((\mathcal V^{B^{\mathrm{simp}}G})^\circ)}
(F_{V,\rho},F_{W,\sigma}).
\]
Finally, the equivalence $\Theta$ of Theorem~\ref{thm:function-model}
induces an equivalence on this mapping space, giving the displayed result.
\end{proof}

\begin{corollary}[strict $n$-representation mapping spaces]
\label{cor:n-representation-mapping-model}
If $F,F'$ lie in
$\bigl(\mathcal V^{B^{\mathrm{simp}}G}\bigr)^{\circ,(n)}$, then their
simplicial mapping object computes the mapping space between the
corresponding objects of $\mathbf{Rep}^{(n)}_k(G)$.
\end{corollary}

\begin{proof}
Apply Theorem~\ref{thm:mapping-spaces-preserved}. The source and target
subcategories in Proposition~\ref{prop:n-representation-function-model}
are full, so restricting to them leaves the relevant mapping spaces
unchanged.
\end{proof}

For arbitrary strict representations, choose projective cofibrant
replacements
\[
QF_{V,\rho}\xrightarrow{\sim}F_{V,\rho},
\qquad
QF_{W,\sigma}\xrightarrow{\sim}F_{W,\sigma}.
\]
Every diagram is already projectively fibrant by
Lemma~\ref{lem:bifibrant-inclusion}, so $QF_{V,\rho}$ and
$QF_{W,\sigma}$ are bifibrant. By
Theorem~\ref{thm:rep-strict-functor} they correspond to strict
representations, say $Q(V,\rho)$ and $Q(W,\sigma)$. The objects
$\Theta QF_{V,\rho}$ and $\Theta QF_{W,\sigma}$ depend only on the
original diagrams up to equivalence, and the theorem gives
\[
\operatorname{Und}\,
\underline{\operatorname{Hom}}_{\mathbf{Rep}}
(Q(V,\rho),Q(W,\sigma))
\simeq
\operatorname{Map}_{\mathcal R_k(G)}
(\Theta QF_{V,\rho},\Theta QF_{W,\sigma}).
\]
Different choices of projective cofibrant replacement give equivalent
objects and equivalent mapping spaces in the model-category localization.
Equivalently, the right side is the derived mapping space of the
original pair; the displayed left side is one point-set computation of
it using projectively cofibrant replacements.

\subsection{The module-theoretic description}\label{sec:module-bridge}

\begin{definition}\label{def:module-category-enriched}

A left $k[G]$-module is a simplicial $k$-vector space $M$ together with
a simplicial map
\[
\mu:k[G]\otimes_{\mathcal V} M\to M
\]
such that for every $n\ge0$,
$\mu(e_n\otimes_{\mathcal V} m)=m$ and
$\mu((ab)\otimes_{\mathcal V} m)=\mu(a\otimes_{\mathcal V}\mu(b\otimes_{\mathcal V} m))$
for all $a,b\in k[G]_n$ and $m\in M_n$. The ordinary category
$k[G]\text{-}\mathbf{Mod}$ has these modules as objects and the
simplicial maps $f:M\to N$ with $f(a\,m)=a\,f(m)$ for all $a\in k[G]_n$
and $m\in M_n$ as morphisms. It is enriched over $\mathcal{V}$: for
modules $M,N$, the mapping object
\[
\underline{\operatorname{Hom}}_{k[G]\text{-}\mathbf{Mod}}(M,N)
\subseteq
\underline{\operatorname{Hom}}_{\mathcal{V}}(M,N)
\]
is the sub-simplicial vector space whose $n$-simplices are the simplicial
maps $f:M\otimes_{\mathcal V} k[\Delta^n]\to N$ satisfying
\[
f(av\otimes_{\mathcal V} u)=a\,f(v\otimes_{\mathcal V} u)
\]
for all $r\ge0$, $a\in k[G]_r$, $v\in M_r$, and
$u\in k[\Delta^n]_r$. Here $k[G]$ acts on the first tensor factor and
acts trivially on $k[\Delta^n]$. In degree zero this condition recovers
the ordinary module morphisms. Composition and identities are inherited
from $\mathcal{V}$.

\end{definition}

{
The mapping subobjects and the inherited operations form a
$\mathcal V$-enrichment, as follows.

For fixed $M,N$ and $n\ge0$, the displayed
condition is a family of $k$-linear equations in $f$, so the admissible
$n$-simplices form a $k$-linear subspace of
$\underline{\operatorname{Hom}}_{\mathcal{V}}(M,N)_n$.
If $f$ is admissible, then its faces
and degeneracies in
$\underline{\operatorname{Hom}}_{\mathcal{V}}(M,N)$ are obtained by
precomposing with the simplicial operators of $\Delta^\bullet$. Because
both the $k[G]$-action and the cosimplicial structure of
$\Delta^\bullet$ are simplicial, each face and degeneracy of $f$
satisfies the same equation. The admissible simplices therefore form a
sub-simplicial vector space.

The composition and identity maps of the
self-enrichment of $\mathcal{V}$ restrict to these subobjects. If $g$
and $h$ are admissible, then for
$\Delta(u)=\sum u_{(1)}\otimes u_{(2)}$,
\[
\begin{aligned}
\bigl(h\circ(g\otimes_{\mathcal V}1)\circ(1\otimes_{\mathcal V}\Delta)\bigr)(am\otimes u)
&=\sum h\bigl(a\,g(m\otimes u_{(1)})\otimes u_{(2)}\bigr)\\
&=\sum a\,h\bigl(g(m\otimes u_{(1)})\otimes u_{(2)}\bigr)\\
&=a\,\bigl(h\circ(g\otimes_{\mathcal V}1)\circ(1\otimes_{\mathcal V}\Delta)\bigr)(m\otimes u),
\end{aligned}
\]
so the composite is admissible.
Let $\varepsilon:k[\Delta^n]\to\underline k$ be the
augmentation sending every simplex to $1$. For each $n$, the unit map
$\underline{k}\to\underline{\operatorname{Hom}}_{\mathcal{V}}(M,M)$
sends $1\in\underline k_n$ to the simplicial map
\[
M\otimes_{\mathcal V} k[\Delta^n]\to M,
\qquad
m\otimes u\mapsto \varepsilon(u)m.
\]
It sends $am\otimes u$ to
$\varepsilon(u)(am)=a(\varepsilon(u)m)$, hence it is admissible; its
$0$-simplex is $\mathrm{id}_M$.

Associativity and the unit equations are inherited from the
self-enrichment of $\mathcal V$.
}

\begin{theorem}\label{thm:module-bridge}
There is an isomorphism of $\mathcal{V}$-enriched categories
\[
\mathbf{Rep}_k^{\mathrm{enr}}(G)
\;\cong\;
k[G]\text{-}\mathbf{Mod}.
\]
Via Theorem~\ref{thm:rep-strict-functor} and this enriched isomorphism,
equip $k[G]\text{-}\mathbf{Mod}$ with the transported projective model
structure from $\mathcal V^{B^{\mathrm{simp}}G}$, and write
$(k[G]\text{-}\mathbf{Mod})^\circ$ for the full enriched
subcategory of bifibrant modules. This transport preserves the
specified weak equivalences, fibrations, and cofibrations. Then there
is an equivalence of
$\infty$-categories
\[
\mathcal{R}_k(G)
\;\simeq\;
N^{\mathrm{hc}}\bigl((k[G]\text{-}\mathbf{Mod})^\circ\bigr).
\]
\end{theorem}

\begin{proof}

First, a strict representation $(V,\rho)$ determines a module
structure by top-simplex evaluation:
\[
\mu_n\Bigl(\sum_i\lambda_i g_i\otimes v\Bigr)
:=\sum_i\lambda_i\,\bigl(\rho_n(g_i)\bigr)_n
       (v\otimes\mathrm{id}_{[n]}),
\qquad g_i\in G_n,\quad v\in V_n.
\]
Here $\rho_n(g_i):V\otimes k[\Delta^n]\to V$ is a simplicial map.
Naturality in both the degree and the parameter simplex makes $\mu$
a simplicial map. The group law, with composition using the parameter
diagonal, gives
\[
\mu_n(e_n\otimes v)=v,\qquad
\mu_n(ab\otimes v)=\mu_n(a\otimes\mu_n(b\otimes v))
\]
first on group-basis elements and then by bilinearity.

Conversely, let $\mu:k[G]\otimes_{\mathcal V}M\to M$ be a module
structure. For $g\in G_n$, define a simplicial map
$\rho_n(g):M\otimes k[\Delta^n]\to M$ by
\[
\bigl(\rho_n(g)\bigr)_r(m\otimes\theta)
=\mu_r(\theta^*g\otimes m),
\qquad \theta:[r]\to[n],\quad m\in M_r,
\]
and extend linearly in the basis $\theta$ of $k[\Delta^n]_r$.
For $\alpha:[s]\to[r]$, simpliciality follows from
\[
\alpha^*\mu_r(\theta^*g\otimes m)
=\mu_s((\theta\alpha)^*g\otimes\alpha^*m).
\]
For $\sigma:[n']\to[n]$, precomposition in the parameter gives
$\sigma^*\rho_n(g)=\rho_{n'}(\sigma^*g)$.
Composition uses the same parameter $\theta$ in both factors, so
\[
\begin{aligned}
\bigl(\rho_n(g)\circ\rho_n(h)\bigr)_r(m\otimes\theta)
&=\mu_r(\theta^*g\otimes\mu_r(\theta^*h\otimes m))\\
&=\mu_r(\theta^*(gh)\otimes m).
\end{aligned}
\]
The unit acts by $m\otimes\theta\mapsto m$ on each basis simplex.
Thus $\rho_n(e_n)$ is the mapping-space unit and
$\rho_n(g^{-1})$ is the inverse of $\rho_n(g)$; hence
$\rho:G\to\operatorname{Aut}(M)$ is a simplicial group homomorphism.

For a given mapping simplex $\rho_n(g)$, naturality recovers every
parameter value from top-simplex evaluation:
\[
\bigl(\rho_n(g)\bigr)_r(m\otimes\theta)
=\bigl(\rho_r(\theta^*g)\bigr)_r(m\otimes\mathrm{id}_{[r]}).
\]
Consequently the two constructions are inverse on strict representations.
Starting from $\mu$, the reverse composite agrees on every group-basis
element and therefore on all of $k[G]$ by linearity.

The enriched categories have the same mapping objects. By
Definition~\ref{def:module-category-enriched}, an $n$-simplex of
$\underline{\operatorname{Hom}}_{k[G]\text{-}\mathbf{Mod}}(M,N)$ is a
simplicial map $f:M\otimes_{\mathcal V} k[\Delta^n]\to N$ satisfying
\[
f(av\otimes_{\mathcal V} u)=a\,f(v\otimes_{\mathcal V} u)
\]
for all $r\ge0$, $a\in k[G]_r$, $v\in M_r$, and $u\in k[\Delta^n]_r$. If $f$ is such
a map, then taking $a=g$ shows that it is equivariant. Conversely, if
$f$ is equivariant, then for $a=\sum_i\lambda_i g_i$,
\[
f(am\otimes_{\mathcal V} u)
=\sum_i\lambda_i\,f(g_i m\otimes_{\mathcal V} u)
=\sum_i\lambda_i\,g_i f(m\otimes_{\mathcal V} u)
=a\,f(m\otimes_{\mathcal V} u),
\]
using linearity of $f$. Thus the mapping objects coincide, and because
composition and identities are inherited from $\mathcal{V}$ in both
cases, the isomorphism is enriched.

The enriched isomorphism transports the projective model
structure and identifies its full bifibrant subcategories. The second
equivalence is therefore the equivalence of
Theorem~\ref{thm:function-model} after applying
Theorem~\ref{thm:rep-strict-functor} and the displayed enriched
isomorphism.

\end{proof}

\begin{corollary}[$n$-representation module model]
\label{cor:n-representation-module-model}
Let $n\ge1$ and let $G$ be a simplicial group. Give
$k[G]\text{-}\mathbf{Mod}$ the transported projective model structure
of Theorem~\ref{thm:module-bridge}. Let
$(k[G]\text{-}\mathbf{Mod})^{\circ,(n)}$ be its full enriched
subcategory on bifibrant modules $M$ satisfying
\[
H_i(NM)=0\qquad(i\ge n).
\]
Then, taking the coherent nerve of the underlying simplicial category,
there is an equivalence
\[
N^{\mathrm{hc}}\bigl((k[G]\text{-}\mathbf{Mod})^{\circ,(n)}\bigr)
\xrightarrow{\ \simeq\ }\mathbf{Rep}^{(n)}_k(G).
\]
It is compatible, up to a canonical natural equivalence, with the
underlying-object functor to $\mathbf{Vect}^{(n)}_k$ and evaluation at
the unique vertex of $\mathbf BG$.
\end{corollary}

\begin{proof}
The enriched isomorphisms of Theorems~\ref{thm:rep-strict-functor}
and~\ref{thm:module-bridge} preserve underlying simplicial vector
spaces. By the definition of the transported model structure, they
identify bifibrant modules with projectively bifibrant diagrams.
They therefore identify the full subcategories specified by
$H_i(NM)=0$ for $i\ge n$. Apply
Proposition~\ref{prop:n-representation-function-model}, including its
evaluation compatibility: both underlying-object and vertex evaluation
take the same simplicial vector space $M$. Membership is determined
by the homology of $NM$ and is preserved by weak equivalences.
\end{proof}

\section{Detection and coefficient obstructions}\label{sec:detection}
Source truncation gives sufficient conditions for restriction to be an
equivalence. We now examine what can be inferred in the reverse direction.
We consider restriction along a specified group map.

\subsection{What $n$-representations retain}
\label{sec:truncated-reconstruction}

For the coefficient target of Definition~\ref{def:n-linear-space},
the first result shows
that the degree-zero group map is still forced to be an isomorphism
when its induced restriction functor is an equivalence.

\begin{theorem}[degree-zero conservativity]\label{thm:A4}
Let $n\ge1$ and let $f:G\to H$ be a map of simplicial groups.  If
\[
f^*:\mathbf{Rep}^{(n)}_k(H)
\longrightarrow
\mathbf{Rep}^{(n)}_k(G)
\]
is an equivalence, then $\pi_0(f)$ is an isomorphism.  Compatibility
with the $n$-representation fiber functors is automatic because $f^*$
is precomposition.
\end{theorem}

\begin{proof}
Let $\mathcal E_0\subset\mathbf{Vect}^{(n)}_k$ be the full replete
subcategory spanned by objects equivalent to vector spaces concentrated
in chain degree zero; equivalently, their normalized homology vanishes
in every positive degree. Theorem~\ref{thm:mapping-homology-formula} shows
that its mapping spaces are $0$-truncated and that
$\mathcal E_0\simeq N(\mathsf{Vect}_k)$ canonically.

For a simplicial group $K$, let $\mathcal D_K$ be the full subcategory of
$\mathbf{Rep}^{(n)}_k(K)$ whose base-point value lies in
$\mathcal E_0$.  Proposition~\ref{prop:one-representations} gives
canonical equivalences, compatible with evaluation,
\[
\mathcal D_K
\simeq
\mathbf{Rep}^{(1)}_k(K)
\simeq
N\bigl(\operatorname{Rep}^{\mathrm{ord}}_k(\pi_0K)\bigr).
\]
Because precomposition does not change the base-point value, $f^*$
preserves these full replete subcategories.  Full faithfulness restricts,
and essential surjectivity restricts as well: a preimage of an object of
$\mathcal D_G$ has an equivalent base-point value and therefore lies in
$\mathcal D_H$.  Hence restriction gives an equivalence fitting into
\[
\begin{CD}
\mathcal D_H @>{f^*}>> \mathcal D_G\\
@V{\simeq}VV @VV{\simeq}V\\
N\operatorname{Rep}^{\mathrm{ord}}_k(\pi_0H)
@>{(\pi_0f)^*}>>
N\operatorname{Rep}^{\mathrm{ord}}_k(\pi_0G).
\end{CD}
\]

The bottom functor is the nerve of the ordinary restriction functor,
so that ordinary functor is an equivalence. Put
$\varphi=\pi_0f:\Gamma\to\Lambda$, $A=k[\Gamma]$, and $B=k[\Lambda]$.
Here $A\text{-}\mathbf{Mod}$ and $B\text{-}\mathbf{Mod}$ contain all left modules.
Restriction along $k[\varphi]$ has the left adjoint
\[
B\otimes_A-:A\text{-}\mathbf{Mod}\longrightarrow B\text{-}\mathbf{Mod}.
\]
Indeed, a $B$-linear map $B\otimes_A M\to N$ is uniquely determined by
the $A$-linear map $m\mapsto F(1\otimes m)$, and conversely this map
extends by $b\otimes m\mapsto bF(1\otimes m)$. Since restriction is an
equivalence, its left adjoint is a quasi-inverse, and the adjunction unit
is an isomorphism at every module. At the regular module it is
\[
A\longrightarrow \operatorname{Res}_\varphi(B\otimes_A A)
\cong\operatorname{Res}_\varphi B.
\]
Under $b\otimes a\mapsto b\,k[\varphi](a)$ this is precisely
$k[\varphi]$. Hence $k[\varphi]$ is a linear isomorphism. If
$\varphi(\gamma)=\varphi(\gamma')$ for distinct group elements, then
the nonzero vector $\gamma-\gamma'$ lies in its kernel. If an element
of $\Lambda$ is missing from the image of $\varphi$, its basis vector
is missing from the linear span of that image. Thus $\varphi$ is both
injective and surjective, and is a group isomorphism.
\end{proof}

\noindent\textbf{Equivalence of module categories.}
Over $\mathbb C$, both $\mathbb C[C_4]$ and
$\mathbb C[C_2\times C_2]$ are isomorphic as algebras to $\mathbb C^4$:
factor $x^4-1$, respectively $x^2-1$ and $y^2-1$, into distinct linear
factors.  Their ordinary representation categories, hence their
$1$-representation categories, are equivalent, although the groups
are not isomorphic. The algebra isomorphisms give equivalences of
all-module categories, with no prescribed compatibility with the
representation tensor product. Such an equivalence therefore does not
determine the group. Theorem~\ref{thm:A4} detects $\pi_0$ using the
restriction functor of a specified group map.

\begin{remark}[source homotopy above degree $n-1$]
Let $\Gamma\ne0$ be abelian.  Since
$\mathbf B K(\Gamma,n)\simeq K(\Gamma,n+1)$ has contractible
$n$-truncation, Theorem~\ref{thm:source-truncation-invariance} gives,
via the specified map $K(\Gamma,n)\to *$, an evaluation-compatible
equivalence
\[
\mathbf{Rep}^{(n)}_k\bigl(K(\Gamma,n)\bigr)
\simeq
\mathbf{Vect}^{(n)}_k.
\]
Local $(n-1)$-truncation of the target makes $n$-representations
insensitive to $\pi_n$ of the source. This example is outside the
homotopy $n$-group range, which stops at $\pi_{n-1}$.
\end{remark}

\subsection{Coefficient homology and failure of detection}
\label{sec:kerodon-direct}

We next consider simplicial groups whose reduced $k$-homology vanishes.
Write
$\mathcal R_k(G)=\operatorname{Fun}(\mathbf BG,\mathcal V^\infty)$
for the untruncated representation category.

\begin{proposition}[coefficient-acyclic sources]
\label{prop:coefficient-acyclic-invisibility}
Let $G$ be a simplicial group such that
\[
\widetilde H_i(|G|;k)=0\qquad(i\ge0).
\]
Then, for the specified map $p:G\to *$, the constant functor
\[
p^*:\mathcal V^\infty
\simeq\mathcal R_k(*)\xrightarrow{\ \simeq\ }\mathcal R_k(G)
\]
is an equivalence with inverse given by vertex evaluation.
For every $r\ge1$, it restricts to an evaluation-compatible equivalence
\[
\mathbf{Vect}^{(r)}_k
\xrightarrow{\ \simeq\ }\mathbf{Rep}^{(r)}_k(G).
\]
\end{proposition}

\begin{proof}
Put $A=k[G]$, with augmentation $\varepsilon:A\to k$ and unit
$\eta:k\to A$. The normalized complex of the underlying simplicial
vector space $A$ is the normalized chain complex of the underlying
simplicial set $G$ with coefficients in $k$. Hence
\[
H_i(NA)\cong H_i(|G|;k).
\]
The hypothesis makes $\varepsilon$ a weak equivalence. Since
$\varepsilon\eta=\mathrm{id}_k$, the unit $\eta$ is also a weak
equivalence.

Over a field, tensoring nonnegative chain complexes preserves
quasi-isomorphisms. The natural $k$-linear shuffle comparison of
Proposition~\ref{prop:enriched-dold-kan-comparison} and
Remark~\ref{rem:dold-kan-monoidal} therefore shows that
$X\otimes_k-$ preserves weak equivalences for every $X\in\mathcal V$.

On simplicial left $A$-modules, let $U$ forget the action,
let $F(X)=A\otimes_kX$ with $A$ acting on the first factor, and let
$T(X)$ have underlying object $X$ with action
$a\cdot x=\varepsilon(a)x$. These functors preserve weak equivalences,
and $UT=\mathrm{id}_{\mathcal V}$.
For each module $M$, there are natural $A$-linear maps
\[
M\xleftarrow{\ \mu_M\ }A\otimes_k UM
\xrightarrow{\ \phi_M\ }T(UM),
\qquad
\mu_M(a\otimes x)=a\cdot x,\quad
\phi_M(a\otimes x)=\varepsilon(a)x.
\]
Their $A$-linearity follows respectively from associativity and
$\varepsilon(ba)=\varepsilon(b)\varepsilon(a)$; the formulas commute
with all simplicial operators. The underlying map of $\phi_M$ is
$\varepsilon\otimes1$, hence a weak equivalence. Moreover,
\[
U\mu_M\circ(\eta\otimes1)=\mathrm{id}_{UM},
\]
so two-out-of-three makes $\mu_M$ a weak equivalence as well.
Thus $\mathrm{id}\leftarrow FU\to TU$ is a natural zigzag of weak
equivalences. After localization, both arrows become natural
equivalences. Together with $UT=\mathrm{id}$, this proves that $T$
and $U$ induce inverse equivalences of $\infty$-categories.

\smallskip\noindent\textbf{Localization of strict modules.}
Let $\mathcal M=A\text{-}\mathrm{Mod}$ carry the projective simplicial
model structure of Theorem~\ref{thm:module-bridge}, and let $W$ be its
weak equivalences. All its objects are fibrant. Functorial factorization
gives an ordinary functor
$Q:\mathcal M_{\mathrm{ord}}\to\mathcal M^\circ_{\mathrm{ord}}$
and a natural weak equivalence $q:Q\to\mathrm{id}$, where
$\mathcal M^\circ$ is the full simplicial subcategory of bifibrant modules;
write $W^\circ$ for its weak equivalences.
Naturality of $q$ and two-out-of-three show that $Q$ preserves $W$.
If $j$ denotes inclusion, the maps $jQ\to\mathrm{id}$ and
$Qj\to\mathrm{id}$ become natural equivalences after localization.
Consequently, $j$ induces an equivalence between the two localizations,
as in \cite[Remark~1.3.4.16]{HA}.
Since every cofibrant module is already fibrant,
\cite[Theorem~1.3.4.20]{HA} then gives
\[
N(\mathcal M_{\mathrm{ord}})[W^{-1}]
\simeq N(\mathcal M^\circ_{\mathrm{ord}})[(W^\circ)^{-1}]
\simeq N^{\mathrm{hc}}(\mathcal M^\circ).
\]
Here the second comparison is induced by the canonical map from the
ordinary nerve to the homotopy coherent nerve: in
\cite[Construction~1.3.4.18 and Lemma~1.3.4.19]{HA}, fibrant replacement
on cofibrant objects can be chosen to be the identity.

\smallskip\noindent\textbf{Comparison with coherent representations.}
The enriched diagram identification in Theorem~\ref{thm:module-bridge}
has an evaluation functor
$B^{\mathrm{simp}}G\times\mathcal M\to\mathcal V$.
Taking homotopy coherent nerves, restricting along the canonical map
$N(\mathcal M_{\mathrm{ord}})\to N^{\mathrm{hc}}(\mathcal M)$,
and currying gives
\[
E:N(\mathcal M_{\mathrm{ord}})
\longrightarrow\operatorname{Fun}(\mathbf BG,\mathcal V^\infty).
\]
A weak equivalence of modules has a weak equivalence as its underlying
map in $\mathcal V$, hence $E$ sends it to a pointwise equivalence.
Equivalences in a functor $\infty$-category are detected at vertices
\cite[Theorem~4.4.4.4, Tag~\href{https://kerodon.net/tag/01DK}{01DK}]{Kerodon},
so $E$ induces a functor $\overline E$ on the localization.
On $\mathcal M^\circ$, this evaluation construction is exactly the
comparison of Theorem~\ref{thm:function-model}.
The displayed localization comparisons therefore show that
\[
\overline E:N(\mathcal M_{\mathrm{ord}})[W^{-1}]
\xrightarrow{\ \simeq\ }\mathcal R_k(G)
\]
is an equivalence.

\smallskip\noindent\textbf{The specified functors.}
Before localization, vertex evaluation of $E(M)$ is $UM$, and
$E(TX)$ is the constant functor with value $X$: in every simplicial
degree, each group basis element acts by the identity, so the strict
action factors through $B^{\mathrm{simp}}*$.
These identities also hold on morphisms and are natural in $M$ and $X$.
Using $N(\mathcal V_{\mathrm{ord}})[W_{\mathcal V}^{-1}]
\simeq\mathcal V^\infty$ from Definition~\ref{def:underlying-infinity},
the universal property of localization extends them to
\[
\operatorname{ev}_*\circ\overline E\simeq U_\infty,
\qquad
\overline E\circ T_\infty\simeq p^*.
\]
This universal property applies to natural transformations as well as
functors \cite[Definition~1.3.4.1]{HA}.
In particular, the maps $E(q_{TX}):E(QTX)\to E(TX)=p^*X$
are natural equivalences, even when $TX$ itself is not projectively
cofibrant. Thus the inverse equivalences induced by $T$ and $U$ above
are precisely the specified constant functor and vertex evaluation.

Finally, finite height is a condition on the homology of the evaluated
object. Both $T$ and $U$ preserve this condition, so their inverse
equivalences restrict to the indicated full subcategories.
\end{proof}

The proposition illustrates \cite[Proposition~4.2.4.9]{HA}:
when the unit map $k\to A$ is an equivalence, forgetting the $A$-action
is an equivalence of categories.

\begin{corollary}[prime-to-characteristic invisibility]
\label{thm:prime-to-character-invisibility}
Let $n\ge2$, let $q$ be a prime invertible in $k$, and put
\[
G=K(\mathbb Z/q,n-1).
\]
Then precomposition along $G\to *$ gives an equivalence
\[
\mathbf{Vect}^{(n)}_k
\simeq
\mathbf{Rep}^{(n)}_k(*)
\xrightarrow{\ \simeq\ }
\mathbf{Rep}^{(n)}_k(G)
\]
compatible with the fiber functors, although
$\pi_{n-1}(G)=\mathbb Z/q\ne0$. The constant functor also gives
an equivalence $\mathcal V^\infty\simeq\mathcal R_k(G)$.
\end{corollary}

\begin{proof}
Write $C_q=\mathbb Z/q$. We prove
$\widetilde H_*(K(C_q,m);k)=0$ for all $m\ge1$.

For $m=1$, use the homogeneous bar model of $BC_q=K(C_q,1)$.
Its unnormalized $k$-chains have basis the diagonal $C_q$-orbits
$[g_0,\ldots,g_a]$, with boundary the alternating sum of deletions
and degeneracies repeating adjacent entries.
The averaging maps
\[
h_a[g_0,\ldots,g_a]
=\frac1q\sum_{g\in C_q}[g,g_0,\ldots,g_a],
\qquad
h_{-1}(1)=\frac1q\sum_{g\in C_q}[g]
\]
are well defined on these orbits by reindexing the sum.
They satisfy $dh+hd=\mathrm{id}$ on the augmented complex.
They also preserve degenerate chains, since a repeated adjacent entry
remains repeated after prepending $g$. Thus they contract the augmented
normalized complex, proving the required reduced homology vanishing.

For $m\ge2$, consider the path fibration
\[
K(C_q,m-1)\simeq\Omega K(C_q,m)
\longrightarrow P K(C_q,m)\longrightarrow K(C_q,m).
\]
The base is simply connected and the total space is contractible.
By induction, the fiber has only $H_0\cong k$. The homology Serre
spectral sequence \cite[Theorem~5.3, p.~527]{HatcherSpectralSequences}
has
\[
E^2_{a,b}
=H_a\bigl(K(C_q,m);H_b(K(C_q,m-1);k)\bigr)
=
\begin{cases}
H_a(K(C_q,m);k),&b=0,\\
0,&b>0.
\end{cases}
\]
All differentials vanish, since their bidegrees are $(-s,s-1)$
for $s\ge2$ and only the bottom row is nonzero.
The first-quadrant spectral sequence converges with a finite filtration
in each total degree to the homology of the contractible path space.
Consequently the base also has only zero-degree homology $k$.
Taking $m=n-1$, Proposition~\ref{prop:coefficient-acyclic-invisibility}
gives both the untruncated and finite-height assertions.
\end{proof}

\noindent\textbf{The regular representation.}
The augmentation $k[G]\to k$ is a weak equivalence, so the regular
representation has homology concentrated in degree zero. Its image
belongs to every finite-height subcategory.
Corollary~\ref{thm:prime-to-character-invisibility} shows that the
representation categories with evaluation fail to distinguish $G$
from the trivial group at every height and in the untruncated case.

\noindent\textbf{Low-dimensional comparison.}
At height two, Heredia--Elgueta
\cite[Theorem~4.3]{HerediaElgueta2016} show that, for a finite $2$-group and
under the corresponding prime-to-characteristic hypotheses, the homotopy
category of representations in Baez--Crans $2$-vector spaces depends only on
$\pi_0$. Their conclusion concerns the homotopy category of a
$2$-categorical coefficient model. For the particular source
$K(\mathbb Z/q,1)$, the equivalence above retains all mapping spaces
and also holds for the untruncated coefficient category.

\subsection{Outlook}
For $n\ge2$ and a prime $q$ invertible in $k$, the sources $K(\mathbb Z/q,n-1)$
and $*$ have different $\pi_{n-1}$ but equivalent representation
categories with evaluation, both at finite height and in the
untruncated case. Thus recovering $Q_nG\simeq\tau_{\le n-1}G$ calls
for restrictions on the source or coefficient targets that detect
these homotopy groups. A separate question is the construction of
a coherent tensor structure on the finite-height coefficient category.

\section*{Appendix A. Notation and source conventions}
\subsection*{A.1 Notation}
\begin{itemize}
\item $\operatorname{const}\Gamma$: the constant simplicial group on an
ordinary group $\Gamma$, with $\Gamma$ in every degree and all faces and
degeneracies equal to the identity (Remark~\ref{rem:homotopy-n-group-strict-models}).
\item $k$: a field; $\mathcal V$: simplicial $k$-vector spaces.
\item $N:\mathcal V\simeq\mathsf{Ch}_{\ge0}(k)$: normalization;
$S^ik$: $k$ concentrated in chain degree $i$.
\item $E_i=k[\Delta^i]/k[\partial\Delta^i]$ for $i\ge0$:
sphere test object; $NE_i\cong S^ik$ and $E_0=k[\Delta^0]$.
\item $\operatorname{Map}(V,W)=
\operatorname{Und}\underline{\operatorname{Hom}}_{\mathcal V}(V,W)$.
\item Within Definition~\ref{def:underlying-infinity},
$\operatorname{ev}_{V,W}$, $m_{V,W,Z}$ and $u_V$ denote evaluation,
composition and the enriched unit for the internal Hom.
\item $\mathsf{sSet}$: simplicial sets; $\mathcal V^\infty$: the
$\infty$-category presented by $(\mathcal V,\operatorname{Map})$.
\item $\mathbf{Vect}^{(n)}_k\subseteq\mathcal V^\infty$:
the full coefficient subcategory on $V$ with $H_i(NV)=0$ for $i\ge n$,
retaining all mapping spaces. Homology concentrated in $0,\ldots,m$
corresponds to $\mathbf{Vect}^{(m+1)}_k$.
\item $\mathbf BG=N^{\mathrm{hc}}(B^{\mathrm{simp}}G)$:
the connected classifying Kan complex;
$\pi_{i+1}(\mathbf BG)\cong\pi_i(G)$ for $i\ge0$, as groups when $i=0$.
\item $\mathbf{Rep}^{(n)}_k(G)=
\operatorname{Fun}(\mathbf BG,\mathbf{Vect}^{(n)}_k)$;
$\omega_G^{(n)}$ is evaluation at the unique vertex.
\item $\mathcal R_k(G)=\operatorname{Fun}(\mathbf BG,\mathcal V^\infty)$:
untruncated coherent representations.
\item $\mathbf{Rep}^{\mathrm{enr}}_k(G)$:
strict enriched representations; the superscript $\circ$ denotes
the full subcategory of projectively bifibrant objects.
\item $\operatorname{Rep}^{\mathrm{ord}}_k(\Gamma)$: ordinary
representations of an ordinary group on all $k$-vector spaces.
\item $(k[G]\text{-}\mathbf{Mod})^{\circ,(n)}$: bifibrant simplicial
modules with $H_i(NM)=0$ for $i\ge n$, with the full enriched Hom.
\item $k[G]$: the simplicial group algebra; $k[G]\text{-}\mathbf{Mod}$:
all simplicial left modules with their equivariant enriched Hom.
\item $h_{\le n}$: homotopy $n$-category; $\pi_{\le n}X$:
the chosen $n$-truncation of a Kan complex.
\item $Q_nG=\operatorname{cosk}^{\mathsf{sGrp}}_n(G|_{\le n})$:
the $(n-1)$-truncated source model.
\item In Proposition~\ref{prop:coefficient-acyclic-invisibility},
$\varepsilon:k[G]\to k$ is the augmentation and $\eta:k\to k[G]$
is the unit. The functors $U$, $F(X)=k[G]\otimes_kX$ and $T$
are respectively forgetting the action, the free module functor,
and assigning the trivial action via $\varepsilon$.
\end{itemize}

\subsection*{A.2 Kerodon citations}

Kerodon citations use the result type, the item number in the online
version checked on 7 October 2026, and its stable Tag.
Each Tag in a citation links to the corresponding statement.
The complete list of cited items is below; stable Tags identify the
statements independently of subsequent numbering changes.

\begin{longtable}{ll}
Current online item & Stable Tag\\
\hline
\endfirsthead
Current online item & Stable Tag\\
\hline
\endhead
Proposition 1.2.5.9 & \href{https://kerodon.net/tag/00MG}{00MG}\\
Theorem 2.4.5.1 & \href{https://kerodon.net/tag/00LJ}{00LJ}\\
Construction 2.4.6.1 & \href{https://kerodon.net/tag/00LX}{00LX}\\
Theorem 2.5.6.1 & \href{https://kerodon.net/tag/00QR}{00QR}\\
Remark 2.5.6.2 & \href{https://kerodon.net/tag/00QS}{00QS}\\
Proposition 2.5.7.8 & \href{https://kerodon.net/tag/00RP}{00RP}\\
Construction 2.5.8.2 & \href{https://kerodon.net/tag/00S2}{00S2}\\
Proposition 2.5.8.3 & \href{https://kerodon.net/tag/00S3}{00S3}\\
Proposition 2.5.8.5 & \href{https://kerodon.net/tag/00S5}{00S5}\\
Construction 2.5.8.6 & \href{https://kerodon.net/tag/00S6}{00S6}\\
Proposition 2.5.8.9 & \href{https://kerodon.net/tag/00S9}{00S9}\\
Theorem 3.2.7.1 & \href{https://kerodon.net/tag/00WV}{00WV}\\
Definition 3.5.5.1 & \href{https://kerodon.net/tag/053M}{053M}\\
Example 3.5.6.8 & \href{https://kerodon.net/tag/054B}{054B}\\
Construction 3.5.6.10 & \href{https://kerodon.net/tag/054D}{054D}\\
Corollary 3.5.6.16 & \href{https://kerodon.net/tag/054L}{054L}\\
Definition 3.5.7.1 & \href{https://kerodon.net/tag/054P}{054P}\\
Proposition 3.5.7.9 & \href{https://kerodon.net/tag/054V}{054V}\\
Example 3.5.7.24 & \href{https://kerodon.net/tag/055B}{055B}\\
Example 3.5.7.26 & \href{https://kerodon.net/tag/055D}{055D}\\
Theorem 4.4.4.4 & \href{https://kerodon.net/tag/01DK}{01DK}\\
Example 4.5.1.13 & \href{https://kerodon.net/tag/01E2}{01E2}\\
Remark 4.6.7.6 & \href{https://kerodon.net/tag/01LF}{01LF}\\
Corollary 4.7.6.24 & \href{https://kerodon.net/tag/06RT}{06RT}\\
Definition 4.7.7.9 & \href{https://kerodon.net/tag/057M}{057M}\\
Corollary 4.7.8.3 & \href{https://kerodon.net/tag/058S}{058S}\\
Definition 4.7.9.1 & \href{https://kerodon.net/tag/0597}{0597}\\
Example 4.7.9.6 & \href{https://kerodon.net/tag/059C}{059C}\\
Proposition 4.7.9.7 & \href{https://kerodon.net/tag/059D}{059D}\\
Corollary 4.7.9.16 & \href{https://kerodon.net/tag/059P}{059P}\\
Definition 4.8.4.8 & \href{https://kerodon.net/tag/01Q0}{01Q0}\\
Corollary 4.8.4.9 & \href{https://kerodon.net/tag/01Q1}{01Q1}\\
\end{longtable}

\section*{Acknowledgments}
I am deeply grateful to my advisor for patient guidance and thoughtful encouragement.
My heartfelt thanks also go to my family for their love, understanding, and unwavering support.
\section*{Tool use}
OpenAI GPT and DeepSeek models, together with Codex, were used to assist this work.
\bibliographystyle{abbrvurl}
\bibliography{references/references}
\enlargethispage{2\baselineskip}
\end{document}